\documentclass[a4paper,11pt]{amsart}

\usepackage[english]{babel}
\usepackage{amssymb, amsmath, amsthm}
\usepackage{color}
\usepackage{graphicx}
\usepackage[final]{hyperref}
\usepackage[a4paper, centering]{geometry}

\allowdisplaybreaks 

\def\uh{\hat{u}}

\def\fe{f_{\epsilon}}

\def\one{\mathbf{1}}

\newcommand{\aac}{<\!\!<}

\newcommand{\mean}[1]{\,-\hskip-1.08em\int_{#1}} 
\newcommand{\meantext}[1]{\,-\hskip-0.88em\int_{#1}} 

\newcommand{\res}{\mathop{\hbox{\vrule height 7pt width .5pt depth 0pt
\vrule height .5pt width 6pt depth 0pt}}\nolimits} 

\def\Resto{\mathcal{R}}

\DeclareMathOperator{\spt}{spt} 
\DeclareMathOperator{\Div}{div}

\DeclareMathOperator{\sign}{sign}

\newcommand{\M}[1]{\mathcal{#1}}    
\renewcommand{\H}{\M{H}}
\renewcommand{\L}{\M{L}}

\newcommand{\R}{\mathbb{R}}

\newcommand{\Haus}[1]{\mathcal{H}^{#1}} 
\newcommand{\Leb}[1]{\mathcal{L}^{#1}} 
\newcommand{\eps}{\varepsilon}

\newcommand{\W}{W}

\def\medint{-\kern  -,375cm\int}

\newcommand{\A}{\boldsymbol A}

\renewcommand{\a}{\boldsymbol a}
\newcommand{\eeta}{\boldsymbol \eta}

\newtheorem{theorem}{Theorem}[section]
\newtheorem*{theoremnonumber}{Theorem}
\newtheorem{lemma}[theorem]{Lemma}
\newtheorem{proposition}[theorem]{Proposition}

\theoremstyle{definition}
\newtheorem{definition}[theorem]{Definition}
\theoremstyle{remark}
\newtheorem{remark}[theorem]{Remark}

\begin{document}

\author[G.~Crasta]{Graziano Crasta}
\address{Dipartimento di Matematica ``G.\ Castelnuovo'', Univ.\ di Roma I\\
P.le A.\ Moro 2 -- I-00185 Roma (Italy)}
\email{crasta@mat.uniroma1.it}
\author[V.~De Cicco]{Virginia De Cicco}
\address{Dipartimento di Scienze di Base  e Applicate per l'Ingegneria, Univ.\ di Roma I\\
Via A.\ Scarpa 10 -- I-00185 Roma (Italy)}
\email{virginia.decicco@sbai.uniroma1.it}
\author[G.~De Philippis]{Guido De Philippis}
\address{Institut f\"ur Mathematik, Universit\"at Z\"urich --
Winterthurerstrasse 190 -- CH-8057 Z\"urich}
\email{guido.dephilippis@math.uzh.ch}

%

\date{September 12, 2014}

\title[Conservation laws with discontinuous flux]{Kinetic formulation and uniqueness for scalar conservation laws
with discontinuous flux}
\begin{abstract}
We prove a uniqueness result for $BV$ solutions of scalar conservation laws with discontinuous flux in several space dimensions.
The proof is based on the notion of kinetic solution and on a careful analysis of the entropy dissipation along the discontinuities of the flux.
\end{abstract}

\maketitle

\section{Introduction}
This paper is concerned with multidimensional scalar conservation laws with discontinuous flux, namely
\begin{equation}
\label{f:equa}
u_t+ \Div \A(x,u)=0,\qquad (t,x)\in (0,+\infty)\times\R^n,
\end{equation}
where the flux function $\A\colon\R^n\times \R\to\R^n$ is possibly discontinuous in the first variable and satisfies some structural assumptions listed in Section~\ref{sec:prel}. These type of equations have attracted a lot of attention in the last years since they naturally arise in several models (for instance  models of traffic flow,
flow in porous media, sedimentation processes, etc.), see \cite{AKR,GNPT,Pan1} and references therein for a more detailed account on the theory.

It is well known that, even for smooth fluxes, in general 
the Cauchy problem 
\begin{equation}
\label{f:cauchy}
\begin{cases}
u_t+ \Div \A(x,u)=0, & (t,x)\in (0,+\infty)\times\R^n,\\
u(0,x) = u_0(x), & x \in \R^n,
\end{cases}
\end{equation}
does  not  admit classical solutions.
On the other hand,  the notion of distributional  solution is too
weak in order to achieve well-posedness, in particular it does not provide uniqueness of the solution.
In this context, the notion of entropy solution turns out to be the right one,
as it has been shown by
Vol'pert \cite{vol} in the $BV$ setting and
by Kruzkov \cite{Kruz} in the $L^{\infty}$
framework.

The classical Kruzkov's approach \cite{Kruz}
completely solves the problem 
of well-posedness
in the case of smooth
fluxes.
Moreover, as it has been shown in recent years,
using a clever change of variables
and the concept of \textsl{adapted entropies},
this approach also works for a restricted class of
discontinuous fluxes
(see \cite{AP,disc0,CDC,Mitr-DCDS,Pan1}).

The case of more general discontinuous fluxes in one space variable has been extensively studied 
(see for example
\cite{AMV,AKR,AnMitr,AP,Cocl,Die,GNPT,disc8,Mitr-DCDS,Mitr} and
references therein). 
In particular, 
it has been pointed out 
(\cite{AMV,AKR,GNPT})
that many different admissibility criteria
generate continuous semigroups of solutions,
and the choice of the right criterion may depend
on the physics of the problem under consideration.
Indeed,
in addition to the classical entropy criteria, one  has to impose some conditions on the behavior of the solutions on the discontinuities of the flux. Roughly speaking different conditions give rise to different criteria.
These different conditions are coded, for example,
in the notion of \textsl{germ} (see \cite{AKR})
or of well-posed Riemann solver (see \cite{GNPT}).
One of the most studied admissibility criterion
is based on the notion of vanishing
viscosity solution.
This will actually  be the  criterion
that we are going to use in this paper, see Remark \ref{rmk:commenti} below.

\medskip

In spite of the intensive study concerning conservation laws with discontinuous fluxes,  
in the multidimensional case there are very few results available in the literature.
A very general existence result
has been obtained by Panov \cite{Pan1}.
On the other hand, a well-posedness result
for a restricted class of fluxes (having only one
regular hypersurface of discontinuity)
has been recently proved by
Mitrovic \cite{Mitr-multi}.

In this  paper we propose an entropy criterion, modeled on the ones introduced in the one-dimensional case in  \cite{AnMitr,Die, Mitr}, which allows to prove  uniqueness of \(BV\) entropy solutions
of the Cauchy problem \eqref{f:cauchy} under mild assumptions on the flux \(\A\). 
More precisely our main result is the following:

\begin{theoremnonumber}   Let \(\A\in L^\infty (\R^n\times \R;\R^n)\) satisfies \((H1)\)--\((H7)\) in Section  \ref{sec:prel}   below and let  \(u_1\) and \(u_2\) be two  \(BV\) entropy solutions of \eqref{f:equa} (see Definition~\ref{d:entrsol}), then
\[
\int_{\R^n}|u_1(T,x)-u_2(T,x)|\,dx\le \int_{\R^n}|u_1(0,x)-u_2(0,x)|\,dx.
\]
\end{theoremnonumber}
Let us mention that  our assumptions are satisfied in the particular case of fluxes of the form  \(\A(x,u)=\widehat \A (k(x),u)\) where \(k\) is in  \(SBV(\R^n; \R^N)\cap L^\infty(\R^n; \R^N) \) and \(\widehat \A\) is smooth, see Remark \ref{lincei}.

On the other hand we do not investigate the issue of existence of \(BV\) entropy solutions. Let us point out  that, unlike the case of fluxes with smooth dependence on \(x\),  when \(\A\) is discontinuous there is no general known result concerning existence of \(BV\) solutions, 
see however Remark \ref{rmk:BV}.

Our method of proof is based on the notion
of kinetic solution
introduced by Lions, Perthame and Tadmor
(see \cite{LPT,Pert1,Pertbook}).
This technique
has been successfully applied 
by Dalibard 
to the case of nonautonomous smooth fluxes
(see \cite{Dal,Dal2}).
More precisely,
we show the equivalence between 
the entropic and the kinetic formulation
of \eqref{f:equa}.
Relying on the nonautonomous chain rule
for $BV$ functions proved in \cite{ACDD}
and on a careful analysis of the entropy
dissipation along the discontinuities of the flux,
we are then able to prove that the 
semigroup of solutions is contractive
in the $L^1$ norm.

\bigskip
The plan of the paper is the following.
In Section~\ref{sec:prel} we state the assumptions
on the flux (see (H1)--(H7) below)
and we discuss some preliminary results.
In Section~\ref{s:formulation}
we introduce our notions of entropy and kinetic solution,
we prove their equivalence
(see Theorem~\ref{equivalence}),
and we state our main uniqueness result
(see Theorem~\ref{thm:uniquenessentropic}).
In Section~\ref{s:prelim} we prove some
preliminary estimates for kinetic solutions,
and, finally,
in Section~\ref{s:uniqueness}
we prove the uniqueness of kinetic solutions
(see Theorem~\ref{uniqueness}).

\bigskip
\textsc{Acknowledgments.}
The authors would like to thank Luigi Ambrosio
and Gianluca Crippa 
for some useful discussions,
and two referees for having improved, with their comments,
	the presentation of the paper.
The authors have been supported by the Gruppo Nazionale per l'Analisi Matematica, 
la Probabilit\`a e le loro Applicazioni (GNAMPA) of the Istituto Nazionale di Alta Matematica (INdAM).

\section{Assumptions on the vector field and the chain rule}\label{sec:prel}
In this section we first survey some useful facts about \(BV\) function that we need in the sequel, we  state our main structural hypotheses  on the vector field (assumptions (H1)-(H7) below)  and prove some consequences of these assumptions.

\subsection{BV functions}
Let us start  recalling  our main notation and preliminary facts on
$BV$ and \(SBV\) functions. A general reference is Chapter 3 of \cite{AFP}, and
occasionally we will give more precise references therein.

We denote by $\Leb{n}$ the Lebesgue measure in $\R^n$ and by
$\Haus{k}$ the $k$-dimensional Hausdorff measure. By Radon measure we mean a
nonnegative Borel measure finite on compact sets. If \(\mu\) is a Radon measure on \(X\) and \(\nu\) is a Radon measure on \(Y\)
we denote by \(\mu\times \nu\) the the product measure on \(X\times Y\), sometimes we will also write \(\mu_x\times \nu_y\). Given a Borel set \(D\) we will denote by \(\mu\res D\) the Radon measure given by \(\mu\res D(B)=\mu(B\cap D)\).
A set \(\Sigma\subset \R^n\) is said to be  countably $\Haus{n-1}$
rectifiable if $\Haus{n-1}$-almost all of $\Sigma$ can be covered by
a sequence of $C^1$ hypersurfaces. Let us recall that if \(\Sigma\subset \R^n\) is  countably $\Haus{n-1}$
rectifiable, then \((a,b)\times \Sigma\subset \R\times \R^n\) is countably $\Haus{n}$
rectifiable and 
\[
\Haus{n}\res (a,b)\times \Sigma=\Leb{1}\res (a,b)\times \Haus{n-1}\res \Sigma,
\]
see \cite[Theorem 3.2.23]{Fed}.

A function  $u\in L^1(\R^n)$ belongs to $BV(\R^n)$ if its derivative
in the sense of distributions is representable by a vector-valued
measure $Du=(D_1u,\ldots,D_nu)$ whose total variation $|Du|$ is
finite, i.e.
$$
\int_{\R^n}u\frac{\partial\phi}{\partial x_i} dx=-\int\phi\,D_iu
\qquad\forall\phi\in C^\infty_c(\R^n),\,\,i=1,\ldots,n
$$
and $|Du|(\R^n)<\infty$. %
%

\noindent {\bf Approximate continuity and jump points.} We say that
$x\in\R^n$ is an approximate continuity point of $u$ if, for some
$z\in\R$, it holds
$$
\lim_{r\downarrow 0}\mean{B_r(x)}|u(y)-z| dy=0.
$$
The number $z$ is uniquely determined at approximate continuity
points and denoted by $\tilde{u}(x)$, the so-called approximate
limit of $u$ at $x$. The complement of the set of approximate
continuity points, the so-called singular set of $u$, will be
denoted by $S_u$.

Analogously, we say that $x$ is a jump point of $u$, and we write
$x\in J_u$, if there exists a unit vector $\nu\in{\bf S}^{n-1}$ and
$u^+,\,u^-\in\R$ satisfying $u^+\neq u^-$ and
$$
\lim_{r\downarrow 0}\mean{B^\pm(x,r)}|u(y)-u^\pm| dy=0,
$$
where $B^\pm(x,r):=\{y\in B_r(x): \pm\langle y-x,\nu\rangle\geq 0\}$
are the two half balls determined by $\nu$. At points $x\in J_u$ the
triplet $(u^+,u^-,\nu)$ is uniquely determined up to a permutation
of $(u^+,u^-)$ and a change of sign of $\nu$; for this reason, with
a slight abuse of notation, we do not emphasize the $\nu$ dependence
of $u^\pm$ and $B^\pm(x,r)$. Since we impose $u^+\neq u^-$, it is
clear that $J_u\subset S_u$, moreover, for every \(u\in BV_{\rm loc}\),
\(\H^{n-1}(S_u\setminus J_u)=0\) and  \(J_u\)    is 
$\Haus{n-1}$ countably rectifiable
(see \cite[Theorem~3.78]{AFP}). 
Finally, we  define the precise representative as 
\begin{equation}\label{eq: def*}
u^*(x)=
\begin{cases}
\tilde u (x)\qquad &x\in \R^n\setminus S_u\\
\big(u^+(x)+u^-(x)\big)/2 &x\in J_u\\
0&\textrm{otherwise.}
\end{cases}
\end{equation}
Note that  \(\H^{n-1}(\{\widetilde u\ne u^*\}\setminus J_u)=0\), in particular, since \(|Du|\aac \H^{n-1}\),  \(|\widetilde Du|(\{\widetilde u\ne u^*\})=0\).

\medskip
\noindent {\bf Decomposition of the distributional
derivative.}
For any oriented and countably $\Haus{n-1}$-rectifiable
set $\Sigma\subset\R^n$ we have
\begin{equation}\label{rappDjSigma}
Du\res \Sigma=(u^+-u^-)\nu_\Sigma\Haus{n-1}\res\Sigma.
\end{equation}

For any $u\in BV(\R^n)$, we can decompose $Du$ as the sum of a
diffuse part, that we shall denote $\widetilde{D}u$, and a jump part,
that we shall denote by $D^ju$. The diffuse part is characterized by
the property that $|\widetilde{D}u|(B)=0$ whenever $\Haus{n-1}(B)$ is
finite, while the jump part is concentrated on a set $\sigma$-finite
with respect to $\Haus{n-1}$. The diffuse part can be then split as
\[
\widetilde{D} u= D^a u+D^c u
\]
where $D^a u$ is the absolutely continuous part with respect to the
Lebesgue measure, while $D^c u$ is the so-called Cantor part. The
density of $Du$ with respect to $\Leb n$ can be represented as
follows
\[
D^a u= \nabla u \, d\Leb n,
\]
where $\nabla u$ is the approximate gradient of $u$, see \cite[Proposition 3.71 and Theorem 3.83]{AFP}. Note also that \(|\widetilde  D u|(\{\widetilde u\ne u^*\})=0\).
The jump part can be easily computed
by taking $\Sigma=J_u$ (or, equivalently, $S_u$) in
\eqref{rappDjSigma}, namely
\[
D^j u=Du\res J_u=(u^+-u^-)\nu_{J_u}\Haus{n-1}\res J_u.
\]
We will say that \(u\in SBV(\R^n)\) if \(D^c u=0\), i.e. if 
\[
Du=(u^+-u^-)\nu_{J_u}\Haus{n-1}\res J_u + \nabla u \Leb{n}.
\]
All these concepts and results extend, mostly arguing component by
component, to vector-valued $BV$ functions, see \cite{AFP} for
details.

\medskip
\noindent {\bf Sets of finite perimeter and coarea formula.}
We will say that a measurable set \(E\) is of finite perimeter if its characteristic function \(\chi_E\) belongs to \(BV(\R^n)\). In this case we denote by
\[
\partial^* E =J_{\chi_E}
\]
its \emph{reduced boundary} (note that this is slightly larger than what it is usually called  reduced  boundary, however it coincides with it up to a  \(\H^{n-1}\) negligible set, see \cite[Chapter 3]{AFP}). Then,
\[
D\chi_E=D^j\chi_E=-\nu_{\partial^* E}\H^{n-1}\res \partial^*E,
\]
in particular \(\chi_E\in SBV\).
We also recall the coarea formula: if \(u\in BV(\R^n)\),  then for \(\Leb{1}\) almost every \(v\in \R\), the characteristic function of the set $\{u > v\}$ belongs to \(BV\) and we have the following equalities between measures:
\[
\begin{split}
Du&=\int_{\R} D\chi_{u>v} dv \\
|Du|&=\int_{\R}|D\chi_{\{u>v\}}| dv=\int_{\R}\H^{n-1}(\partial^* \{u>v\}) dv\,.
\end{split}
\]
The following lemma relates the  pointwise behavior  of  \(\chi_{\{u>v\}}\) to the pointwise behavior of \(u\),  see \cite[Lemma 2.2]{DCFV2} for the proof.
\begin{lemma} 
For any function $u\colon \R^n\to\R$ and any $v\in\R$, let $\varphi_{u, v}(x)=\chi_{\{u>v\}}(x)$. 
If $u\in BV(\R^n)$, then for $\L^1$-a.e. $v\in\R$ there
exists a Borel set $N_v\subset\R^n$, with $\H^{n-1}(N_v)=0$, such that 
the following relation holds:
\[
\varphi_{u^{\pm}, v} (x) = \varphi_{u,v}^{\pm}(x),
\qquad
\forall x\in\R^n\setminus N_v.
\]
\end{lemma}

\subsection{Structural assumptions on the vector field}
  Let \(\A\in L^\infty (\R^n\times \R;\R^n)\) be such that:
\begin{enumerate}
\item[(H1)] There exists a set \(\mathcal C_{\A}\) with \(\Leb{n}(\mathcal C_{\A})=0\) such that 
 \(\A(x,\cdot)\in C^1(\R)\) for every \(x\in \R^n\setminus  \mathcal C_{\A}\) and \(\A(\cdot,v)\in SBV(\R^n)\) for every \(v\in \R^n\).
\item[(H2)]There exists a constant \(M\) such that
 \[
 |\partial_v \A(x, v)|\le M\qquad  \forall \, x\in \R^n\setminus \mathcal C_{\A},\quad v\in \R.
 \]

\item[(H3)]There exists a modulus of continuity \(\omega\) such that 
\[
|\partial_v\A (x, u)-\partial _v \A (x,w)|\le \omega(|u-w|)\qquad \forall\, x\in \R^n\setminus\mathcal  C_{\A},\quad u\,,w \in \R.
\]
\item[(H4)]There exists a function \(g_1\in L^1(\R^n)\)  such that  
\[
\begin{split}
|\nabla_x \A(x,u)-\nabla_x \A(x,w)|&\le g_1(x)|u-w|
\qquad
 \forall\, x\in \R^n\setminus\mathcal  C_{\A},\quad u\,,w \in \R,
\end{split}
\] 
where \(\nabla_x \A(x,v)\)
denotes the approximate gradient of the map \(x\mapsto \A(x,u)\).
 \item[(H5)] The measure
\[
\sigma := \bigvee_{u\in \R} |D_x \A(\cdot, u)|
\]
satisfies \(\sigma(\R^n)<\infty\). Here \(D_x \A(\cdot, u)\) is the distributional gradient of the map  \(x\mapsto \A(x,u)\) (which is a measure since \(\A(\cdot ,u)\in BV\)) and \(\bigvee\) denotes the least upper bound in the space of nonnegative Borel measures, see \cite[Definition 1.68]{AFP}.
\end{enumerate}

\subsection{Chain rule and fine properties of \(\A\)}\label{subsec:chainrule}

 Assumptions (H1)-(H5) imply that  \(\A\) satisfies the hypothesis of \cite{ACDD}. Let us summarize some consequence of this fact. If we define 
\[
\mathcal N=\Big\{x\in \R^n: \liminf_{r \to 0} \frac{\sigma(B_r(x))}{r^{n-1}}>0\Big\},
\]
then \(\mathcal N\) is a \(\H^{n-1}\)  rectifiable set. In the sequel we shall assume that:
\begin{itemize}
\item[(H6)]
$\H^{n-1}(\mathcal N)<+\infty$.
\end{itemize}

In our context, Theorem~2.2 in \cite{ACDD} reads as follows: For every \(u\in BV(\R^n;\R)\) the composite function \(w(x)=\A(x ,u(x))\) belongs to  \(BV(\R^n;\R^n)\) with
\begin{equation}\label{eq:stima}
|Dw|\le \sigma+M|Du|
\end{equation}
and 
\begin{align}
\label{eq:chainrule1}
\widetilde {D }w&=\nabla_x \A (x,\widetilde u(x))\Leb{n}+\partial_v \A (x,\widetilde u(x))\otimes \widetilde {D} u
\\
\label{eq:chainrule2}
D ^j w&=\big[\A^+(x,u^+(x))-\A^{-}(x,u^-(x))\big]\otimes \nu_{J_u\cup \mathcal N}\,\mathcal H^{n-1}\res (J_u\cup \mathcal N).
\end{align}
Here the functions \(\A^{\pm}(x,v)\) are defined for \(\H^{n-1}\) almost every \(x\in \mathcal N\cup J_u\) and \emph{every} \(v\in \R\) as 
\begin{equation}\label{eq:defApm}
\A^{\pm}(x,v)=\lim_{r\to 0} \mean{B^\pm _r(x)} \A(y,v) dy,
\end{equation}
in particular
\[
\A^+(x,v)=\A^-(x,v)\qquad \textrm{\(\H^{n-1}\)-a.e. \(x\in J_u\setminus \mathcal N\) and every \(v\in \R\).} 
\]
Note that by interchanging the derivative with the integral, applying Ascoli-Arzel\`a Theorem and taking into account  (H2) and (H3), we can deduce as in \cite[Section 3]{ACDD} that the map \(v\mapsto \A^{\pm}(x,v)\) is \(C^1\) for almost every \(x\in \mathcal N\) with derivative given by \(\partial_v \A^{\pm}(x,v)=(\partial_v \A(x,v))^{\pm}\) (it is part of the statement the fact that this last quantity is well defined). In particular for  \(\H^{n-1}\)-a.e. \(x\in\mathcal N\) and all \(u\,,w \in \R\) we have
\begin{equation}\label{modulopm}
|(\partial_v\A)^{\pm} (x, u)-(\partial _v \A )^{\pm}(x,w)|\le \omega(|u-w|)\,.
\end{equation}

 In the same way, see \cite[Section  3]{ACDD},   for \(\H^{n-1}\) almost every point of \(\R^n\setminus \mathcal N\) and every \(v\in \R\)  there exists the limit
\[
\widetilde \A(x,v)=\lim_{r\to 0} \mean{B_r(x)} \A(y,u)\,dy,
\]
and the map \(v\mapsto \widetilde \A(x,v)\) is continuously differentiable with \(\partial_v \widetilde \A(x,v)= \widetilde {\partial_v \A}(x,v)\). Moreover for \(\mathcal L^{n}\)-a.e. \(x\)  and all \(u\,,w \in \R\)
\begin{equation}\label{modulopreciso}
|\widetilde{\partial_v\A} (x, u)-\widetilde{\partial _v \A }(x,w)|\le \omega(|u-w|)\,.
\end{equation}

We conclude this section with the following simple remark. Thanks to (H4) and the previous discussion, the functions
\[
\boldsymbol B_h(x,v)=\frac{\A(x,v+h)-\A(x,v)}{h}
\]
satisfy
\begin{equation}\label{eq:unibound}
|\boldsymbol B_h(x,v)|\le M, \qquad |D_x \boldsymbol B_h (\cdot, v)|\leq g_1(x)\L^n +M\,\H^{n-1}\res \mathcal N\,,
\end{equation}
where the first inequality follows from (H2). 
Using (H3) one can show that  \( \boldsymbol B_h(x,v)\to \partial_v \A(x,v)\) for almost every \(x\) and \emph{every} \(v\in \R\), see \cite[Section 3]{ACDD} for similar arguments. From \eqref{eq:unibound}  we deduce that \(\partial_v \A(\cdot, v)\in SBV\). Let us now consider the decomposition
\[
D_x  \partial_v \boldsymbol A(\cdot,v)=\nabla_x  \partial_v \A(\cdot,v)\Leb{n}+\big((\partial_v\A)^+(\cdot,v)-(\partial_v\A)^-(\cdot,v)\big)\otimes \nu _{\mathcal N}\, \H^{n-1}\res \mathcal N.
\]
According to the  discussion  below equation \eqref{eq:defApm} we have \((\partial_v\A)^\pm=\partial_v (\A^\pm)\) for \(\H^{n-1}\res \mathcal N\) almost every \(x\) and every \(v\). Moreover,  by (H4) for every \(v\)  the family
\[
h \mapsto \frac{\nabla_x \A(x,v+h)-\nabla_x \A(x,v)}{h}\,,\qquad
h\neq 0,\]
is weakly compact in \(L^1\) and every one of its cluster points has to coincide with \(\nabla_x \partial_v \A(x,v)\), i.e.
\[
\nabla_x \partial_v \A(x,v)= w_{L^1}-\lim_{h \to 0} \frac{\nabla_x \A(x,v+h)-\nabla_x \A(x,v)}{h}.
\]
Let us now define for every Borel and bounded function \(\varphi\) the maps
\[
 h_1(v):= \int \varphi(x)\nabla_x \A(x,v)\, dx,   \quad h_2(v):= \int_{\mathcal N} \varphi(x)(\A^+(x,v)-\A^-(x,v))\otimes \nu _{\mathcal N} \,d\H^{n-1}(x).
\]
By the previous discussion we then see that \(h_1,h_2\) are Lipschitz continuous and everywhere differentiable with derivatives given by
\begin{equation}\label{eq:scambio}
\begin{split}
\frac{dh_1(v)}{dv}&= \int \varphi(x)\nabla_x \partial_v \A(x,v)\, dx
\\
 \frac{dh_2(v)}{dv}&=\int_{\mathcal N} \varphi(x)\Big(\partial_v \A^+(x,v)-\partial_v \A^-(x,v)\Big)\otimes \nu _{\mathcal N} \,d\H^{n-1}.
 \end{split}
\end{equation}
However in the sequel we will also need the following assumption, ensuring the continuity of the map \(v\to dh_1/dv\):
\begin{enumerate}
\item[(H7)]
 There exist a \(L^1\) function \(g_2\) and a modulus of continuity \(\omega\) (which we can assume without loss of generality to be equal to the one appearing  in (H3)) such that 
\begin{equation*}
\big|\nabla_x \partial_v \A(x,u)-\nabla_x \partial_v \A(x,w)|\le g_2(x) \omega(|u-w|)\qquad \forall \, u, w\in \R.
\end{equation*}
\end{enumerate}
With this assumption we have
\begin{lemma}\label{lemma scambio}
Let \(\A\) satisfy (H1)--(H7), then there exists a set \(\widetilde {\mathcal C}_{\A}\) with \(\Leb{n}(\widetilde {\mathcal C}_{\A})=0\) such that  every \(x\in \R^n\setminus \widetilde {\mathcal C}_{\A}\) is a Lebesgue point for \(x\mapsto \nabla_x  \A(x,v)\), \(x\mapsto \nabla_x \partial_v \A(x,v)\) and for any such \(x\) the map \(v\mapsto \widetilde{\nabla_x \A}(x,v)\) is \(C^1\) with derivative given by \(\widetilde{\nabla_x  \partial_v \A}(x,v)\)
\end{lemma}

\begin{proof} Let \(U\subset \R\) be a countable dense set and let
\[
\widetilde {\mathcal C}_{\A}=\bigcup_{u\in U} \R^n\setminus S_{\A(\cdot,u)}\cup (\R^n\setminus S_{g_1}) \cup (\R^n\setminus S_{g_2}),
\]
which clearly satisfies \(\Leb{n}(\widetilde{ \mathcal C}_{\A})=0\). By arguing as in \cite[Section 3]{ACDD} and using (H4) and (H7) we see that for every \(x\in \R^n \setminus \widetilde{ \mathcal C}_{\A}\) the limits
\[
\widetilde{\nabla_x \A}(x,v)= \lim_{r\to 0}
\mean{B_r(x)}\nabla_x \A(y,v)dy\qquad \widetilde{\nabla_x \partial_v \A}(x,v) =\lim_{r\to 0}\mean{B_r(x)}\nabla_x \partial_v \A(y,v)dy
\]
exist for every \(v\in \R\). By the first equality in \eqref{eq:scambio} the map
\[
v\mapsto h_r(v)= \mean{B_r(x)}\nabla_x \A(y,v)\,dy
\]
are differentiable with derivative given by \(h'_r(v)=\meantext{B_r(x)}\nabla_x \partial_v \A(y,v)\,dy\). Since \(x\) is a Lebesgue point for \(g_2\), thanks to (H7) this  is a family of equi-continuous functions in \(v\) converging to \(\widetilde{\nabla_x \partial \A}(x,v)\). It is now a standard argument to see that \(\widetilde{\nabla_x \A}(x,v)=\lim_{r} h_r(v)\) is \(C^1\) with derivative given by \(\lim_r h_r'(v)=\widetilde{\nabla_x \partial_v \A}(x,v)\).
\end{proof}

\begin{remark}\label{lincei}
Let us  point out that our hypotheses include (and actually are modeled on) the case \(\A(x,u)=\widehat \A (k(x),u)\) where \(k\in SBV(\R^n; \R^N)\cap L^\infty(\R^n; \R^N) \), \(\H^{n-1}(J_{k})<+\infty\) and 
\(\widehat \A\in C^1(\R^N\times \R,\R^n)\cap{\rm Lip} (\R^N\times \R,\R^n) \).
\end{remark}

\section{Formulation of the problem}
\label{s:formulation}

\subsection{Entropic formulation}
We consider the following scalar conservation law
\begin{equation}\label{f:scalar}
u_t +\Div \A(x,u)=0\,,
\end{equation}
where
$\A\colon\R^n\times\R\to\R^n$ satisfies the structural assumption (H1)--(H7).

\begin{definition}[Convex entropy pair]
We say that $(S,\eeta)$ is a convex entropy pair
if $S\in C^2(\R)$ is a convex function,
and $\eeta = (\eta_1, \ldots, \eta_n)$ is defined by
\begin{equation}\label{f:eta}
\eta_{i}(x,v):=\int_{0}^{v}\partial_v A_{i}(x,w)S'(w)dw\,,
\qquad i=1,\ldots,n.
\end{equation}
\end{definition}
In the above definition and in the sequel, 
\[A_i=\A\cdot e_i\,,
\]
 are the components of \(\A\).

Note that according to the discussion in Section \ref{subsec:chainrule}, \(\eeta(\cdot, v) \in SBV(\R^n;\R^n)\) for every \(v\in \R\) and its distributional derivative is given by
\[
\begin{split}
D_x\eeta(\cdot,v)&=\left(\int_0^v \nabla_x\partial_v \A(x,w)S'(w)dw\right) \Leb{n}\\
&+\left(\int_0^v \big(\partial_v \A^+(x,w)-\partial_v \A^-(x,w)\big)S'(w)dw \right) \otimes \nu_{\mathcal N} \,d\mathcal H^{n-1}\res \mathcal N\,.
\end{split}
\]

\begin{definition}[Entropy solutions]\label{d:entrsol}
A function 
\[
u\in C([0,T]; L^1(\R^n)) \cap L^{\infty}((0,T)\times \R^n)\cap L^1((0,T);BV(\R^n))
\]
is an entropy solution of \eqref{f:scalar}
if 
$u$ is a solution to \eqref{f:scalar} in the sense of distributions,
and there exists a  (everywhere defined) Borel representative
$\uh$ of $u$ with \(|\uh(t,x)|\le \|u\|_{\infty}\) such that, for every 
convex entropy pair $(S,\eeta)$,
one has
\begin{equation}\label{f:distr}
\begin{split}
\partial_t S(u)&+\Div \big(\eeta(x,u)\big)\\
&-\Div \eeta(x,v)\Big|_{v=\uh(t,x)}+S'(\uh)\Div \A(x,v)\Big|_{v=\uh(t,x)}\le 0
\end{split}
\end{equation}
in  the distributional sense. Here, by  \(\Div \A(x,v)\big|_{v=\uh(t,x)}\)
we mean the measure whose action on a bounded and Borel function \(\varphi=\varphi(t,x)\) is given by 
\begin{multline}\label{def:misuraresto}
\sum_{i=1}^n \int _0^T dt \int_{\R^n} \varphi(t,x) \nabla_{i} A_i(x,\uh(t,x)) \, dx\\
 +\sum_{i=1}^n \int _0^T dt \int_{\mathcal  N} \varphi(t,x)\big(A_i^+(x,\uh(t,x))-A_i^+(x,\uh(t,x))\big)\nu^{i}_{\mathcal N}\, d\H^{n-1}(x),
\end{multline}
and the same for \( \Div \eeta(x,v)\big|_{v=\uh(t,x)}\).
\end{definition}

\begin{remark}\label{rmk:commenti}
Some comments about our definition of entropy solution
are in order.
We shall see in Section~\ref{s:analysis} that,
if $x\in\mathcal{N}$, then our entropy condition
characterizes the (closure of the) so-called vanishing viscosity germ defined in equation (5.4) of \cite{AKR}
(see \eqref{f:EC3} below).
This characterization has also been used by
Diehl
(see Condition $\Gamma$ in \cite{Die}),
Mitrovic \cite{Mitr} and Andreianov \& Mitrovic \cite{AnMitr}.
In this sense, despite the appearance of a somewhat
arbitrary Borel representative $\uh$,
this definition seems to be the natural extension
to our general framework of the conditions
cited above.
We also note that, 
although in Definition~\ref{d:entrsol} we require
$\uh$ to be defined everywhere,
it will be  clear by our arguments below, that it is enough to define $\uh$
\(\Leb{1}\times(\Leb{n}+\sigma)\)-a.e.
\end{remark}

\begin{remark}
\label{rmk:BV}
By definition, if $u$ is an entropy solution 
to~\eqref{f:scalar}, then $u(t,\cdot)\in BV(\R^n)$ for
every $t$.
For regular fluxes  with respect to the space variables, it is well known that  the
entropy  solution to the Cauchy problem with initial data
$u_0\in BV(\R^n)$ remains in $BV$ for all times
(see \cite{Kruz,CoMeRo}). On the other hand, it is not clear when such regularity
has to be expected for entropy solutions
when the flux is discontinuous in the space variables.
A relevant situation for which we can expect
$BV$-regularity of solutions is the case $n=1$, at least
for a class of fluxes widely studied in the
literature 
(see for example \cite{ADGV,GNPT} and references therein).
For instance, if we assume that the flux is piecewise constant
and that, at every point of discontinuity,
it satisfies an appropriate version of the crossing condition, 
one can show the existence of a $BV$ solution
of the Cauchy problem
(see \cite[Theorem~2.13]{ADGV} and \cite[Lemma~9]{GNPT}).
\end{remark}

With these definitions at hand we can now restate our main result:

\begin{theorem}\label{thm:uniquenessentropic}   Let \(\A\in L^\infty (\R^n\times \R;\R^n)\) satisfies \((H1)\)--\((H7)\)  and let  \(u_1\) and \(u_2\) be two entropy solutions of \eqref{f:scalar}, then
\begin{equation}\label{eq:contrattivo}
\int_{\R^n}|u_1(T,x)-u_2(T,x)|\,dx\le \int_{\R^n}|u_1(0,x)-u_2(0,x)|\,dx.
\end{equation}
\end{theorem}

\begin{remark}Since we are dealing with bounded solutions, one can suitably localize assumptions (H1)--(H7) in the ``vertical'' variable \(v\), see for instance \cite{ACDD}. Furthermore one can also localize in the space variable \(x\) by just requiring for instance  that  \(\A(\cdot,v)\)  belongs to \(SBV_{\rm loc}\) and similarly for \(u\). Since this will not add any new ideas to the proof below,  we leave this generalization to the interested reader. Moreover by exploiting standard techniques in the context of hyperbolic equation the following localized version of \eqref{eq:contrattivo} can be easily obtained from the proof of 
Theorem~\ref{thm:uniquenessentropic}:
\[
\int_{B_R(0)}|u_1(T,x)-u_2(T,x)| \,dx\le \int_{B_{R+VT}(0)}|u_1(0,x)-u_2(0,x)|\,dx \qquad \forall\, R\geq 0,
\]
where \(V:=\|\A\|_{\infty}\).
\end{remark}

Since  \(u\in L^1\big((0,T);BV(\R^n)\big)\), by the discussion in Section~\ref{subsec:chainrule}, see  \eqref{eq:stima} in particular, we have that
\[
\A(x,u(t,x))\in L^1\big((0,T);BV(\R^n;\R^n)).
\]
By arguing for instance as in \cite[Theorem 4.3.1]{Daf} we then deduce the following:
\begin{lemma} 
Let \(u\) be an entropy solution of \eqref{f:scalar}, then 
\[
u\in BV((0,T)\times\R^n).
\]
\end{lemma}
Obviously we can think of  \(\A(\cdot, v)\) as a function in \(SBV((0,T)\times\R^n;\R^n)\) constant in time, hence  equations \eqref{f:scalar} and \eqref{eq:chainrule2} give the following Rankine-Hugoniot conditions
\begin{multline*}
 \big(u^+-u^-)\nu_t +\big(\A^+(x,u^+)-\A^-(x,u^-)\big)\cdot \nu_x=0\\
  \qquad\textrm{for \(\H^{n}\)-a.e. \((t,x)\in ((0,T)\times \mathcal N)\cup J_u\),}
\end{multline*}
where \(\nu=(\nu_t,\nu_x)\) is the normal to the  \(\H^{n}\) rectifiable  set   \(((0,T)\times \mathcal N)\cup J_u\subset (0,T)\times \R^n\). 
In particular, since for \(\H^{n}\) almost every \((t,x)\) in \(((0,T)\times \mathcal N)\cap J_u\) we have \(\nu=(0,\nu_\mathcal N)\),
we obtain
\begin{equation}
\label{eq:RH2}
A^+(x,u^+)=A^-(x,u^-) \qquad\textrm{for \(\H^{n}\)-a.e. \((t,x)\in ((0,T)\times \mathcal N)\cap J_u\),}
\end{equation}
where we have introduced the notation
\begin{equation}\label{shortahand1}
A^{\pm}(x,v)=\A^\pm(x,v)\cdot \nu_{\mathcal N}\qquad v\in \R.
\end{equation}

\subsection{Analysis of the entropy condition on discontinuities.}
\label{s:analysis}
Let $u$ be an entropy solution of \eqref{f:scalar}. Thanks to \eqref{eq:chainrule2} and \eqref{def:misuraresto}, on  \((0,T)\times \mathcal N\) the entropy inequality  \eqref{f:distr} reads as 
\begin{equation}
\label{f:EC}
\eta^+(x,u^+)-\eta^+(x,\uh)+A^+(x,\uh) S'(\uh) \leq
\eta^-(x,u^-)-\eta^-(x,\uh)+A^-(x,\uh) S'(\uh),
\end{equation}
\(\H^n\) almost everywhere on \((0,T)\times \mathcal N\) and  we have dropped the dependence on \((t,x)\) from \(u^\pm\) and \(\uh\) in order to simplify the notations. Here we understand that  $u^{\pm} = \widetilde u$ if \(x\notin J_u\) and  we are using the short hand notations  \eqref{shortahand1} and 
\[
\eta^{\pm}(x, v) := \eeta^{\pm}(x,v)\cdot \nu_{\mathcal N}(x).
\]
Thanks to the definition of \(\eeta\), \eqref{f:eta}, and \eqref{eq:scambio}, condition \eqref{f:EC} can be rewritten as
\begin{multline}\label{f:EC0.5}
A^+(x, u^+) S'(u^+) - \int_{\uh}^{u^+} A^+(x,v) S''(v)\, dv\\
 \le A^-(x, u^-) S'(u^-) - \int_{\uh}^{u^-} A^-(x,v) S''(v)\, dv\,.
\end{multline}
By a standard approximation argument,  we can plug in \eqref{f:EC0.5} the  Kruzkov--type entropies
\[
S(v) := |v-c|,
\qquad
\eta_i(v) := \int_0^v \partial_v A_i(x,w) \sign(w-c)\, dw,
\]
where $c\in\R$ is a constant. We then obtain \(\H^n\) almost everywhere on \((0,T)\times \mathcal N\)
\begin{multline}\label{f:EC1}
A^+(x,u^+) \sign(u^+-c)  - 2\sign(u^+ - \uh) A^+(x,c) \one_{(\uh, u^+)}(c) 
\\  \leq
A^-(x,u^-) \sign(u^--c)-2\sign(u^- - \uh) A^-(x,c) \one_{(\uh, u^-)}(c). 
\end{multline}
Here and in the sequel,  for \(a,b\in \R\)
the symbol $\one_{(a,b)}$ will denote the characteristic function
of the open  interval $I(a,b)$ with endpoints \(a\) and \(b\), i.e.\ \(I(a,b)=(a,b)\) if \(a<b\) or \(I(a,b)=(b,a)\) if \(b < a\).

We shall now derive as a consequence of \eqref{f:EC1} some inequalities which will be useful in the last part of the proof of Theorem \ref{thm:uniquenessentropic}. To this end, let us now fix a point \((t,x)\in (0,T)\times \mathcal N\) for which \eqref{f:EC1} and \eqref{eq:RH2} are  valid and let   $I(u^-,u^+)$ be  the open  interval with endpoints $u^-(t,x)$ and $u^+(t,x)$. Let us consider the following two cases:

\begin{itemize}
\item[(a)] $\uh\in I(u^-,u^+)$ 
\item[(b)] $\uh\notin I(u^-,u^+)$\,.
\end{itemize}
In the case (a), taking into account the
Rankine--Hugoniot condition \eqref{eq:RH2},
by (\ref{f:EC1}) we get
\begin{multline}\label{f:EC2}
 \left[\sign(u^+-c) - \sign(u^--c)\right] A^+(u^+)
\\ 
 \leq
2 \sign(u^+-u^-) \left[A^+(c) \one_{(\uh, u^+)}(c)
+ A^-(c) \one_{(\uh, u^-)}(c)\right]
\end{multline}
where we have dropped the dependence on \(x\) from \(A\). This condition gives information only for $c\in I(u^-,u^+)$:
\begin{multline}\label{f:EC3}
 \sign(u^+-u^-) A^+(u^+)
\\
 \leq
\sign(u^+-u^-) \left[A^+(c) \one_{(\uh, u^+)}(c)
+ A^-(c) \one_{(\uh, u^-)}(c)\right],
\qquad c\in I(u^-,u^+).
\end{multline}
As particular cases, taking
$c\nearrow \uh$ and $c\searrow \uh$ we get
\[
\sign(u^+ - u^-) A^+(u^+) \leq
\sign(u^+ - u^-) A^{\pm}(\uh).
\]

Similarly, in the case (b), taking again into account the
Rankine--Hugoniot condition \eqref{eq:RH2},
by (\ref{f:EC1}) we get
\begin{multline}\label{f:EC5}
 \left[\sign(u^+-c) - \sign(u^--c)\right] A^+(u^+)
\\  \leq
2 \sign(u^+-\uh) \left[A^+(c) \one_{(\uh, u^+)}(c)
- A^-(c) \one_{(\uh, u^-)}(c)\right]\,,
\qquad
c\not\in I(u^-,u^+).
\end{multline}
We will  now analyze conditions \eqref{f:EC2} and \eqref{f:EC5} in all the possible cases of different positions of $u^-$, $u^+$, $\uh$ and $c$ (see \cite{Mitr} for a similar analysis).  We list all the cases for reader's convenience. First of all we remark that if $c\geq \max\{u^+,u^-,\uh\}$ or $c\leq \min\{u^+,u^-,\uh\}$, then, by the Rankine--Hugoniot condition \eqref{eq:RH2},  condition \eqref{f:EC1} does not give any information. Therefore we list all other possible cases:
\bigskip

Case 1: $u^+\leq u^-$.

Subcase 1a: $u^+\leq u^-\leq \uh$
\begin{equation} \label{Subcase1a(i)}
(i)\quad u^+\leq u^-\leq c\leq \uh\quad \Rightarrow\quad A^+(c)\leq A^-(c)
\end{equation}
\begin{equation} \label{Subcase1a(ii)}
(ii)\quad u^+\leq c\leq u^-\leq \uh\quad \Rightarrow\quad A^+(c)\leq A^+(u^+)\,.
\end{equation}

\bigskip
Subcase 1b: $u^+\leq \uh\leq u^-$
\begin{equation} \label{Subcase1b(iii)}
(iii)\quad u^+\leq \uh\leq c\leq u^-\quad \Rightarrow\quad A^-(c)\leq A^-(u^-)
\end{equation}
\begin{equation} \label{Subcase1b(iv)}
(iv)\quad u^+\leq c\leq \uh\leq u^-\quad \Rightarrow\quad A^+(c)\leq A^+(u^+)\,.
\end{equation}

\bigskip
Subcase 1c: $\uh\leq u^+\leq  u^-$
\begin{equation} \label{Subcase1c(v)}
(v)\quad \uh\leq u^+\leq c\leq  u^-\quad \Rightarrow\quad A^+(c)\leq A^+(u^+)
\end{equation}
\begin{equation} \label{Subcase1c(vi)}
(vi)\quad \uh\leq c\leq u^+\leq  u^-\quad \Rightarrow\quad A^-(c)\leq A^+(c)\,.
\end{equation}

\bigskip
Case 2: $u^-\leq u^+$.

Subcase 2a: $u^-\leq u^+\leq \uh$
\begin{equation} \label{Subcase2a(i)}
(i)\quad u^-\leq u^+\leq c\leq \uh\quad \Rightarrow\quad A^+(c)\leq A^-(c)
\end{equation}
\begin{equation} \label{Subcase2a(ii)}
(ii)\quad u^-\leq c\leq u^+\leq \uh\quad \Rightarrow\quad A^-(u^-)\leq A^-(c)\,.
\end{equation}

\bigskip
Subcase 2b: $u^-\leq \uh\leq u^+$
\begin{equation} \label{Subcase2b(iii)}
(iii)\quad u^-\leq \uh\leq c\leq u^+\quad \Rightarrow\quad A^+(u^+)\leq A^+(c)
\end{equation}
\begin{equation} \label{Subcase2b(iv)}
(iv)\quad u^-\leq c\leq \uh\leq u^+\quad \Rightarrow\quad A^-(u^-)\leq A^-(c)\,.
\end{equation}

\bigskip
Subcase 2c: $\uh\leq u^-\leq  u^+$
\begin{equation} \label{Subcase2c(v)}
(v)\quad \uh\leq u^-\leq c\leq  u^+\quad \Rightarrow\quad A^+(u^+)\leq A^+(c)
\end{equation}
\begin{equation} \label{Subcase2c(vi)}
(vi)\quad \uh\leq c\leq u^-\leq  u^+\quad \Rightarrow\quad A^-(c)\leq A^+(c)\,.
\end{equation}

\subsection{Kinetic formulation}
Let us define the function $\chi\colon\R^2\to\R$,
\begin{equation}\label{defchi}
\chi(v,u):=
\begin{cases}
1
&\text{if $v<u$},\\
1/2
&\text{if $v=u$},\\
0
&\text{if $u<v$}.
\end{cases}
\end{equation}
Note that if \(u\in BV((0,T)\times \R^n)\) then for almost every \(v\in \R\) the function $(t,x) \mapsto \chi(v,u(t,x))$ belongs to $SBV_{loc}$.

Let  $\A$ satisfy conditions (H1)--(H7), and let us define
\begin{align*}
&a_i(x,v) := \partial_v A_i(x,v)&&\forall i=1,\dots,n\\
&a_{n+1}(\cdot,v) :=-\sum_{i=1}^n{\partial_i A_i}(\cdot,v)\,.
\end{align*}
We remark that, for every $v\in\R$, $a_i(\cdot,v)$ is a $BV$ function, while $a_{n+1}(\cdot,v)$ is a Radon measure, moreover according to (H5)
\begin{equation}\label{compact}
\sup_{v}|a_{n+1}(\cdot,v)|(\R^n)\le \sigma(\R^n)<+\infty\,.
\end{equation}
Let us denote by
\begin{equation*}
\a(\cdot,v):=(a_1(\cdot,v)\,\L^n_x,\dots,a_n(\cdot,v)\,\L^n_x,a_{n+1}(\cdot,v)).
\end{equation*}
Note that \(\a\) is a Radon measure and that
\(
\Div_{x,v} \a=0.
\)

\begin{definition}[Kinetic solutions]
\label{d:kinetic}
A function 
\[
u\in C([0,T]; L^1(\R^n)) \cap L^{\infty}((0,T)\times \R^n)\cap L^1((0,T);BV(\R^n))
\]
is a kinetic solution of \eqref{f:scalar}
if 
$u$ is a solution to \eqref{f:scalar} in the sense of distributions,
and there exists a  (everywhere defined) Borel representative
$\uh$ of $u$ with \(|\uh(t,x)|\le \|u\|_{\infty}\) and a positive measure
$m(t,x,v)$  with 
$m((0,T)\times\R^{n+1}) < +\infty$ such that  the function
$(t,x,v)\mapsto \chi(v, \uh(t,x))$ satisfies
\begin{equation}\label{f:kin}
\partial_t \chi(v,\uh(t,x))+\Div_{x,v}[\a(x,v)\chi(v,\uh(t,x))]
=\partial_{v}(m(t,x,v))
\end{equation}
in the sense of distributions.
\end{definition}


Our first results establishes the equivalence between Definitions~\ref{d:entrsol} and~\ref{d:kinetic}.
\begin{theorem}\label{equivalence}
Let  
\[
u\in C([0,T); L^1(\R^n)) \cap L^{\infty}((0,T)\times \R^n)\cap L^1((0,T);BV(\R^n)).
\]
Then $u$ is an entropy solution to \eqref{f:scalar} if and only if it is a kinetic solution to \eqref{f:scalar}.
\end{theorem}

\proof 
We divide the proof in two steps.

\medskip
\noindent
{\it Step 1}. Let \(u\) be a kinetic solution and let  \(S\in C_c^\infty(\R)\).
By testing   \eqref{f:kin} with \(\phi(t)\varphi(x)S'(v)\), we then obtain
 \begin{equation}\label{eq:1}
\begin{split}
&\int_{(0,T)\times \R^n\times \R} \phi'(t)\varphi(x)S'(v) \chi(v,\uh(t,x))dtdxdv\\
&+\sum_{i=1}^n \int_{(0,T)\times \R^n\times \R} \phi(t)  \partial_{i} \varphi(x)S'(v) a_{i}(x,v) \chi(v,\uh(t,x))dtdxdv\\
&+\int_{(0,T)\times \R^n\times \R} \phi(t)   \varphi(x)S''(v)  \chi(v,\uh(t,x))\,d a_{n+1}(x,v) dv dt \\
&=\int_{(0,T)\times \R^n\times \R} S''(v)\varphi(x)\phi(t)dm(t,x,v)\,.
\end{split}
\end{equation} 
Now, since \(S\) is compactly supported,
\begin{equation}\label{eq:2}
\int_{\R} S'(v) \chi(v,\uh(t,x)) dv = \int_{-\infty}^{\uh(t,x)} S'(v)dv=S(\uh(t,x))
\end{equation}
and, for \(i=1,\dots,n\), 
\begin{equation}\label{eq:3}
\int_{\R} S'(v) a_{i}(x,v) \chi(v,\uh(t,x)) dv =\int_{-\infty}^{\uh(t,x)} S'(v) \partial_v A_{i}(x,v) dv=\overline{\eta}_i(x,\uh(t,x)),
\end{equation}
where
\[
\overline{\eeta}(x,v) :=
\int_{-\infty}^{v} S'(w) \partial_v \A(x,w) dw
= \eeta(x,v) + 
\int_{-\infty}^{0} S'(w) \partial_v \A(x,w) dw
\]
by the definition of \(\eeta\) in \eqref{f:eta}.
Moreover
\begin{equation}\label{eq:4}
\begin{split}
\int_{ \R^n\times \R} &\varphi(x)S''(v)  \chi(v,\uh(t,x))\,d a_{n+1}(x,v)\, dv\\
= {} & -\sum_{i=1}^n \int_{\R^{n}\times \R}  \varphi(x) S''(v) \chi(v,\uh(t,x)) \nabla_{i} A_i(x,v)\, dv\, dx \\
& - \int_{\mathcal N \times \R}  \varphi(x) S''(v) \chi(v,\uh(t,x)) \big(A^+(x,v)-A^-(x,v)\big)\,  dv\, d\mathcal H^{n-1}(x)\,,
\end{split}
\end{equation}
where we have used the short hand notation \eqref{shortahand1}. Now by the discussion in Section \ref{subsec:chainrule}, the map \(v\mapsto A^{\pm}(x,v)\) is \(C^1\) for \(\H^{n-1}\) almost every \(x\in \mathcal N\) with derivative given by  \(v\mapsto \partial_v A^{\pm}(x,v)\), hence for any such \(x\) 
\begin{equation}\label{eq:5}
\begin{split}
\int_{\R}&S''(v) \chi(v,\uh(t,x)) \Big(A^+(x,v)-A^-(x,v)\Big)  dv
\\ = {} &
\int_{-\infty}^{\uh(t,x)}S''(v) \big(A^+(x,v)-A^-(x,v)\big)  dv
\\ = {} & 
S'(\uh(t,x))\big(A^+(x,\uh(t,x))-A^-(x,\uh(t,x))\big)
\\ & - \int_{-\infty}^{\uh(t,x)}S'(v) \big(\partial_v A^+(x,v)-\partial_v A^-(x,v)\big)\,dv
\\ = {} &
S'(\uh(t,x))\big(A^+(x,\uh(t,x))-A^-(x,\uh(t,x))\big)- \big(\overline{\eta}^+(x,\uh(t,x))-\overline{\eta}^-(x,\uh(t,x))\big)
\end{split}
\end{equation}
where we are using for $\overline{\eeta}$ the same
convention \eqref{shortahand1} used for $\A$ and $\boldsymbol{\eta}$.
 In the same way by Lemma \ref{lemma scambio}, for almost every \(x\in \R^n\) the map \(v\mapsto \nabla_{i}A_i(x,v)\) is \(C^1\) with derivative given by \(  \nabla_{i}\partial_v A_i(x,v)\), hence for any such \(x\)
\begin{equation}\label{eq:6}
\begin{split}
  \int_{\R}&S''(v) \chi(v,\uh(t,x)) \nabla_i A_i(x,v)  dv\\
&= S'(\uh(t,x))\nabla_i A(x,\uh(t,x))- \int_{-\infty}^{\uh(t,x)}S'(v) \nabla_i \partial_v A_i(x,v)dv\\
&= S'(\uh(t,x))\nabla_i A(x,\uh(t,x))- \nabla_i \overline{\eta}_i(x,\uh(t,x)).
\end{split}
\end{equation} 
Combining \eqref{eq:1}, \eqref{eq:2}, \eqref{eq:3}, \eqref{eq:4}, \eqref{eq:5} and \eqref{eq:6}  we deduce that if \(u\) is a kinetic solution of \eqref{f:scalar}, then  for every function \(S\in C_c^\infty(\R)\) we have
\[
\begin{split}
\partial_t S(u)&+\Div \big(\overline\eeta(x,u)\big)\\
&-\Div \overline\eeta(x,v)\Big|_{v=\uh}+S'(\hat u)\Div \A(x,v)\Big|_{v=\uh}=-\int S''(w) dm(\cdot,\cdot, w)
\end{split}
\]
in the sense of distribution. 
We now note that $\overline{\eeta}-\eeta$ is a function
of the $x$ variable only, hence the
above equation implies that
\begin{equation}\label{eq:7}
\begin{split}
\partial_t S(u)&+\Div \big(\eeta(x,u)\big)\\
&-\Div \eeta(x,v)\Big|_{v=\uh}+S'(\hat u)\Div \A(x,v)\Big|_{v=\uh}=-\int S''(w) dm(\cdot,\cdot, w)
\end{split}
\end{equation}
for every $S\in C_c^{\infty}(\R)$.
Using the very same approximation argument
of the second step of the proof of Theorem~3 in \cite{Dal},
we conclude that \eqref{eq:7} holds for every
convex function $S$ of class $C^2$.
Since \(m\geq 0\),
this fact  
implies that \(u\) is an entropy solution of \eqref{f:scalar}.

 \medskip
 \noindent
 {\it Step 2}. Let \(u\) be an entropy solution of \eqref{f:scalar}, let us define the distribution
\begin{equation}\label{eq:8}
\begin{split}
m(t,x,v)&=\partial_t \int_0^v \chi(w,u(t,x))dw\\
&+\sum_{i=1}^n \partial_i\left\{ \int_0^v a_i(x,w) \chi(w,u(t,x))dw \right\}+a_{n+1}(x,v)\chi(v,\uh(t,x)). 
\end{split}
\end{equation}
Clearly \eqref{f:kin} is satisfied in the sense of distributions, hence to conclude that \(u\) is a kinetic solution we only have to show that \(m\) is a positive measure with  \(m((0,T)\times \R^n\times \R)<\infty\). First note that by testing \eqref{eq:8} with \(\phi(t,x)\psi(v)\) with \(\spt \psi \cap [-\|u\|_{\infty},\|u\|_{\infty}]=\emptyset\) and recalling that \(|\uh|\le \|u\|_{\infty}\) we obtain  that, as a  distribution,
\[
m(t,x,v)=\partial_t u+\Div\A(x,u)=0
\quad\text{outside}\ [0,T]\times\R^n\times  [-\|u\|_{\infty},\|u\|_{\infty}],
\]
that is \(\spt m\subset [0,T]\times\R^n\times  [-\|u\|_{\infty},\|u\|_{\infty}]\). 
We want to show that \(m\) is a positive measure, to this end note that  by the same computations of Step 1, \(u\) satisfies \eqref{eq:7} for every \(S\in C_c^\infty(\R)\).
If \(\psi(v)\geq 0\) is a positive test function, since \(m\) vanishes outside  \([0,T]\times\R^n\times  [-\|u\|_{\infty},\|u\|_{\infty}]\) we can construct a  function \(S\in C_c^{\infty}\)  which is convex on the range of \(\uh\) and for which
\[
\int \psi(v) dm(t,x,v)= \int  S''(v) dm(t,x,v).
\]
By \eqref{eq:8} and since \(u\) is an entropy solution we then deduce that
\[
\int \psi(v) dm(t,x,v)\geq 0\qquad \forall \psi\geq 0,
\]
hence \(m\) is a positive measure.  Moreover choosing \(\phi=1\) on \([-\|u\|_{\infty},\|u\|_{\infty}]\) so that \(S(v)=v^2/2\) on the range of \(\uh\) and integrating \eqref{eq:7} we get
\[
\int_{(0,T)\times \R^{n}\times \R} dm(t,x,v)\le \frac1 2 \int_{\R^n} |u(0,x)|^2\,dx+\int_{(0,T)\times \R^{n}\times \R} d|a_{n+1}|\, dt\,.
\]
By \eqref{compact} we finally conclude.
\qed

\section{Preliminary estimates}\label{s:prelim}
In this section we shall prove some preliminary estimates
for approximate solutions that
will be used in Section~\ref{s:uniqueness}.

Given an entropy solution $\uh$,
let us define the function $f(t,x,v) := \chi(v, \uh(t,x))$ and let us consider a regularization of $f$ with respect to
the $v$ variable. 
More precisely,  let $\varphi\in C^\infty_c([-1/2,1/2])$ be  such that \(\varphi(w)=\varphi(-w)\), \(\varphi\geq 0\)  and  \(\int \varphi=1\). Let
\[
\varphi_\epsilon(v)=\frac 1 \epsilon \varphi\left(\frac{v}{\epsilon}\right)
\]
denote the standard family of
mollifiers. If we define
\[
\fe(t,x,v) := (f(t,x,\cdot)\ast\varphi_{\epsilon})(v)=\big(\chi (\cdot,u )\ast \varphi_{\epsilon}\big)(v)\Big|_{u=\uh(t,x)},
\]
then we have the following
\begin{proposition}\label{fepsilon}
The function $\fe$ satisfies the following equation:
\begin{equation}\label{feps}
\partial_t \fe(t,x,v)+
\Div_{x,v}[\a(x,v)\fe(t,x,v)]
=\partial_{v}(m_\epsilon(t,x,v))+r_\epsilon\,,
\end{equation}
where
\[
m_{\epsilon}(t,x,v) := (m(t,x,\cdot)\ast\varphi_{\epsilon})(v),
\]
and the commutator
\[
r_{\epsilon} :=  -\Div_{x,v } \big((\a f)\ast\varphi_{\epsilon}\big)+\Div_{x,v } \big(\a f_{\epsilon}\big)
\]
is a measure on \((0,T)\times \R^n\times \R\) such that 
\begin{equation}\label{eq:zero}
\lim_{\epsilon\to 0} |r_\epsilon|\big((0,T)\times \R^n\times (c,d)\big)=0
\end{equation}
for every $c,d\in\R$.
\end{proposition} 

\begin{proof}
The fact that \(\fe\) satisfies \eqref{feps} is evident, hence we only have to verify the last part of the statement. To this end note let us write, with obvious notations, \(r_\epsilon\) as 
\[
\begin{split}
r_\eps&=\Div_x (\partial_v \A f_\epsilon)-\Div_x ((\partial_v \A f)_\epsilon)\\
&+\partial_v (a_{n+1} f_\epsilon)-\partial_v ((a_{n+1} f)_\epsilon):=r_{1,\epsilon}+r_{2,\epsilon}
\end{split}
\]
and let us show that both \(r_{1,\epsilon}\) and \(r_{2,\epsilon}\)  are measure satisfying \eqref{eq:zero}.

\medskip
\noindent
\(\bullet\){\em Computation of \(r_{1,\epsilon}\):}\\
 If we test the distribution \(r_{1,\epsilon}\) with functions \(\phi(t,x,v)=\phi_1(t)\phi_2(x)\phi_3(v)\) and we apply Fubini Theorem we obtain
\begin{equation}\label{r1}
\begin{split}
\langle r_{1,\epsilon},\phi\rangle
 & =-\int_{\R}\int_0^T dv dt \, \phi_1(t) \phi_3(v)\int dw\, \varphi(w) 
\\ & \qquad \times \int_{\R^n} dx\, D_{x} \phi_2(x) \cdot \big(\partial_v \A(x,v)-\partial_v \A(x,v-\epsilon w)\big) f(t,x,v-\epsilon w)\\
&=\int_{\R}\int_0^T dv dt \, \phi_1(t) \phi_3(v)\int dw\, \varphi(w)  
\\ & \qquad \times \int_{\R^n} dx\,  \phi_2(x)\Div_x\Big[ \big(\partial_v \A(x,v)-\partial_v \A(x,v-\epsilon w)\big) f(t,x,v-\epsilon w)\Big].
\end{split}
\end{equation}
Recall that, by the coarea formula,  for \(\Leb{2}\) almost every \((v,w)\) the set \(\big\{(t,x)\in (0,T)\times \R^n: u(t,x)>v-\epsilon w\big\}\) is of finite perimeter in \((0,T)\times \R^n\)  and that denoting by  \(J_{v,w}=\partial^*\big\{(t,x)\in (0,T)\times \R^n: u(t,x)>v-\epsilon w\big\}\) we have
\[
D_x f (t,x,v-\epsilon w)=D^j_x f (t,x,v-\epsilon w)=\nu_{J_{v,w}}\H^{n-1}\res J_{v,w}.
\]
 Exploiting the Leibniz rule in \(BV\), \cite[Exercise 3.97]{AFP}, 
 we then find that for \(\Leb{2}\) almost every \((v,w)\)
 \begin{equation}\label{r11}
 \begin{split}
 &\Div\Big[ \big(\partial_v \A(x,v)-\partial_v \A(x,v-\epsilon w)\big) f(t,x,v-\epsilon w)\Big]\\
 &=\sum_{i=1}^n \big(\nabla_i \partial_v A_i(x,v)- \nabla_i \partial_v  A_i(x,v-\eps w) \big)f(t,x, v-\eps w) \Leb{1}_t\times \Leb{n}_x\\
 &\quad+ \big( \partial_v A^+(x,v)- \partial_v  A^+(x,v-\eps w) \big)f^+(t,x, v-\eps w) \Leb{1}_t\times  \mathcal H^{n-1}\res\mathcal N\\
 &\quad- \big( \partial_v A^-(x,v)- \partial_v  A^-(x,v-\eps w) \big)f^-(t,x, v-\eps w)\Leb{1}_t\times  \mathcal H^{n-1}\res\mathcal N\\
 &\quad+\big( \widetilde{\partial_v \A}(x,v)-\widetilde{ \partial_v  \A}(x,v-\eps w) \big)\cdot \nu_{J_{v,w}}
 \\
& \qquad\times\big(f^+(t,x, v-\eps w)-f^-(t,x, v-\eps w)\big)\mathcal H^{n}\res J_{v,w}\setminus 
 ((0,T)\times\mathcal N)\,,
 \end{split}
 \end{equation}
 where we are using the notation \eqref{shortahand1}. According to    (H3), (H7) and \eqref{modulopm}, \eqref{modulopreciso} we then obtain  from \eqref{r1} and \eqref{r11} that \(r_{1,\epsilon}\) is a measure and that
 \begin{equation}\label{i1}
 \begin{split}
 |r_{1,\epsilon}|& ((0,T)\times \R^n\times (c,d))\\
 &\le T\omega(\epsilon) (d-c)
  \|g_2\|_{L^1(\R^n)}
  +2T\omega(\epsilon) 
 (d-c) \H^{n-1}(\mathcal N)\\
 &+\omega(\epsilon)\int\varphi(w) d w\int_{\R} 
 \H^{n}(J_{v,w}) d v.
\end{split}
 \end{equation} 
 Clearly the first two terms on the right hand side of \eqref{i1} go to zero as \(\epsilon\) goes to zero. For what concerns the last one we notice that by the coarea formula
 \[
 \begin{split}
 \int \varphi(w) d w \int_{\R} 
 \H^{n}(J_{v,w}) dv &\le \int \varphi(w) dw \int_{\R} 
 \H^{n}\big(\partial^*\big\{ u(t,x)>v\big\} \big) d v \\
 &\le |Du|((0,T)\times \R^n).
 \end{split}
 \]
Hence also the third term in right hand side of \eqref{i1} goes to zero.

\medskip
\noindent
\(\bullet\){\em Computation of \(r_{2,\epsilon}\):}\\
A computation similar to the previous one shows that \(r_{2,\epsilon}\) is given by the following measure: 
\begin{equation}\label{r2}
\begin{split}
r_{2,\epsilon}&=\Big\{\sum_{i=1}^n\int \nabla_i \partial_v A_i(x,v) f(t,x,v-\epsilon w) \varphi(w)dw  \\
&\quad\quad -\sum_{i=1}^n\int\big(\nabla_i A_i(x,v)-\nabla_i A_i(x,v-\eps w)\big) f(t,x,v-\epsilon w) \frac{\varphi'(w)}{\epsilon} dw \Big\}\,\Leb{1}_t \times \Leb{n}_x\times \Leb{1}_v \\
&\quad+\Big[\int \big(\partial_vA^+(x,v)- \partial_v A^-(x,v)\big) f(t,x,v-\epsilon w) \varphi(w)dw\\
&\quad\quad -\int\big(A^+(x,v)-A^+(x,v-\epsilon w)\big) f(t,x,v-\epsilon w) \frac{\varphi'(w)}{\epsilon} dw\\
&\quad\quad +\int \big(A^-(x,v)-A^-(x,v-\epsilon w)\big) f(t,x,v-\epsilon w) \frac{\varphi'(w)}{\epsilon} dw \Big]\,\,\Leb{1}_t \times \H^{n-1}\res \mathcal N\times \Leb{1}_v , 
\end{split}
\end{equation}
where  we are again using the convention  \eqref{shortahand1}. Note that by (H2) and the discussion in Section~\ref{subsec:chainrule}, 
\(|A^{\pm}(x,v)-A^\pm(x,v-\epsilon w)|\le M \epsilon |w|\) for \(\H^{n-1}\)-a.e.\ \(x\in \mathcal N\) and that,  by (H4), \(\big|\nabla_i A_i(x,v)-\nabla_i A_i(x,v-\eps w)\big|\le g_1(x) \epsilon |w| \). 
Hence if we can show that the term in curly brackets (respectively in square bracket) in \eqref{r2} goes to zero as \(\epsilon\) goes to zero for \(\Leb{1}_t \times \Leb{n}_x\times \Leb{1}_v \) almost every \((t,x,v)\) (respectively for \(\Leb{1}_t \times \H^{n-1}\res \mathcal N\times \Leb{1}_v\) almost every \((t,x,v)\)), an application of the Lebesgue dominated convergence theorem applied with respect to the measure  
\(\Leb{1}_t \times \Leb{n}_x\times \Leb{1}_v \) (respectively \(\Leb{1}_t \times \H^{n-1}\res \mathcal N\times \Leb{1}_v\)) will imply that \(r_{2,\epsilon}\) satisfies \eqref{eq:zero}. 
To this end let us write 
\begin{equation}\label{r22}
\begin{split}
&\int \nabla_i \partial_v \nabla A_i(x,v) f(t,x,v-\epsilon w) \varphi(w)dw \\
&\quad-\int\big(\nabla_i A_i(x,v)-\nabla_i A_i(x,v-\eps w)\big) f(t,x,v-\epsilon w) \frac{\varphi'(w)}{\epsilon} dw\\
&=\int \big\{\nabla_i \partial_v A_i(x,v)\varphi(w)-\nabla_i \partial_v A_i(x,v)w\varphi'(w)\big\} f(t,x,v-\epsilon w)dw\\
&\quad-\int\Bigg\{\frac{\nabla_i A_i(x,v)-\nabla_i A_i(x,v-\eps w)}{\epsilon}- w\, \nabla_i  \partial_v  A(x,v) \Bigg\} f(t,x,v-\epsilon w) \varphi'(w) dw.
\end{split}
\end{equation}
Since by Lemma \ref{lemma scambio}, the map \(v\to \nabla_iA_i(x,v)\) is differentiable for almost every \(x\) with derivative given by \( \nabla_i  \partial_v A_i(x,v)\), by the fundamental theorem of calculus and (H7) for every such \(x\) we can estimate the second integral in the right hand side of \eqref{r22} by
 \[
  \omega(\epsilon)\, g_2(x) \int |w\phi'(w)| dw ,
  \]
 which goes to zero as \(\epsilon \to 0\). For what concerns the first term we notice that \(f(t,x,v-\eps w)\to f(t,x,w)\) for almost every \((t,x,v)\) since
\begin{equation*}
\Leb{1}\times\Leb{n} \Big(\{(t,x): \uh(t,x)=v\}\Big)=0
\end{equation*}
for all but countably many \(v\). Hence the first term in the right hand side of \eqref{r22} goes to
\[
 \nabla_i \partial_v A_i(x,v)f(t,x,v) \int [\varphi(w)-w\varphi'(w)] dw=0.
\]
The term in square bracket in \eqref{r2} can be treated in the same way this time  using that \(v\mapsto A^\pm(x,v)\) is differentiable for \(\H^{n-1}\) almost every \(x\in \mathcal N\) with derivative given by \(\partial_v A^\pm(x,v\)) and that 
\begin{equation*}
\Leb{1}\times\H^{n-1}\res \mathcal N  \Big(\{(t,x): \uh(t,x)=v\}\Big)=0
\end{equation*}
for all but countably many \(v\).
\end{proof}

\bigskip
The next step concerns the derivation of the evolution equation 
satisfied by $\fe^2 := (\fe)^2$.
In order to simplify the notation, we define
the differential operator $L$ by
\[
Lg(t,x,v):=\partial_t g(t,x,v)+
\Div_{x,v}[\a(x,v)g(t,x,v)].
\]

\begin{lemma} \label{fepsilonquadro}
The function $f_\epsilon^2$ satisfies the following equation
\[
\begin{split}
Lf^2_\epsilon
& = 2f_\epsilon^*Lf_\epsilon
+\Resto[\fe]
\end{split}
\]
that is
\begin{multline}\label{fepsquadro1}
\partial_t f^2_\epsilon(t,x,v)+
\Div_{x,v}[\a(x,v)f^2_\epsilon(t,x,v)]
\\ = 2(f_\epsilon(t,x,v))^*[\partial_{v}(m_\epsilon(t,x,v))+r_\epsilon]
+\Resto[\fe]
\end{multline}
where \(f_\epsilon^*\) is defined as in \eqref{eq: def*} and 
\begin{equation}\label{resto0}
\begin{split}
\Resto[\fe] := \Big\{&
[\left(\partial_v A\right)^{+}-\left(\partial_v A\right)^{-}]
\big(\fe^+ - \fe\big)
\big(\fe - \fe^-\big)\\
&+\big(\fe^+ - \fe+\fe^- - \fe\big) [A^{+} - A^{-}] \partial_v\fe\,\Big\} \L^1_t\times \H^{n-1}\res \mathcal N\times \L^1_v\,.
\end{split}
\end{equation}
Here we are using the convention \eqref{shortahand1}.
\end{lemma}

\begin{proof}
Note that \(f_\eps(t,x,v)\) is a \(BV_{loc}\) function with respect to all its variables. By the chain rule 
\begin{equation*}
\widetilde {\partial_t} f_\epsilon^2(t,x,v)=2f^*_\epsilon(t,x,v)\widetilde{\partial_t} f_\epsilon(t,x,v)\,,\qquad
\widetilde{D_x}  f_\epsilon^2(t,x,v)=2f^*_\epsilon(t,x,v)\widetilde{D_x}  f_\epsilon(t,x,v)
\end{equation*}
for \(\H^{n+1}\) almost every \((t,x,v)\in \R^{n+2}\setminus J_{f_\epsilon}\). 
Since \(v\mapsto f_\eps(t,x,v)\)  is \(C^1\) we also have (recall that \(\uh(t,x)\) is everywhere defined)
\[
\partial_v f_\epsilon^2(t,x,v)=2f_\epsilon(t,x,v)\partial_v f_\epsilon(t,x,v)\,,\\
\]
for every \((t,x,v)\). Moreover 
\[
\chi(\cdot,u)\ast\varphi_\epsilon=\frac{1}{\epsilon}\int_{-\infty}^u \varphi\Big(\frac{v-w}{\epsilon}\Big)dw
\]
is a smooth function of \(u\). According to this \(J_{f^2_\epsilon}=J_{f_\epsilon}=J_u\times \R\). Hence, recalling that \eqref{shortahand1} is in force,
\begin{equation}\label{resto1}
\begin{split}
&L\fe^2 -2(\fe)^*L\fe \\
&=\Big\{\big(\partial_v A^+ (f^2_\epsilon)^+-\partial_v A^- (f^2_\epsilon)^-\big )-2\fe^* \big(\partial_v A^+ f_\epsilon^+-\partial_v A^- f_\epsilon^-\big )\Big\}\L^1_t\times \H^{n-1}\res \mathcal N\times \L^1_v\\
&\quad+ \widetilde{ \partial_v \A}\cdot \nu_{J_{f_\epsilon}}\Big\{\big((f^2_\epsilon)^+-(f^2_\epsilon)^-\big) -2\fe^*\big(f_\epsilon^+-f_\epsilon^-\big)\Big\}
\H_{t,x,v}^{n+1}\res (J_{f_\epsilon}\setminus 
[(0,T)\times\mathcal N\times\R])\\
&\quad+\Big\{\big( A^+ -A^- \big)\big(2\fe-2\fe^*\big)\partial_v \fe\Big\}\L^1_t\times \H^{n-1}\res \mathcal N\times \L^1_v\\
&\quad+\Big\{\big((\partial_v A)^+ -(\partial_v A)^- \big)\big(\fe-2\fe^*\big)\fe\Big\}\L^1_t\times \H^{n-1}\res \mathcal N\times \L^1_v,
\end{split}
\end{equation}
where we have used that \(\partial_v (A^{\pm})=(\partial_v A)^\pm\) for \(\H^{n-1}\) almost every \(x\) in \(\mathcal N\) and that \(\fe^*=\fe\), for \(\Leb{n}+|\widetilde {D_{t,x}}\fe|\) almost every \(x\), see Section \ref{subsec:chainrule}. Since \((\fe^2)^\pm=(\fe^\pm)^2\) we have
\[
(f^2_\epsilon)^+-(f^2_\epsilon)^- -2\fe^*\big(f_\epsilon^+-f_\epsilon^-\big)=0\,,
\]
hence the second line in the right hand side of \eqref{resto1} vanishes. A simple algebraic computation now shows that \eqref{resto1} reduces to \eqref{resto0}.
\end{proof}

Let us now consider, for \(R>0\), the  following test function $\psi_R\in C^\infty_c(\R^{n+1})$,
$\psi_R\geq 0$ defined by
\begin{equation}\label{f:testf}
\psi_R(x,v)=\theta\left(\frac{x}{R}\right)\phi_R(v),
\end{equation}
with $\theta\in C^\infty_c(\R^n)$\,, $\phi_R\in C^\infty(\R)$\,, $\phi_R(v)=1$ if $|v|\leq R$\,,
$\phi_R(v)=0$ if $|v|\geq R+1$\,, $\theta(x)=1$ if $|x|\leq 1$ and $|\phi'_R|\leq 2$\,.
\bigskip

\begin{lemma}\label{prop.19}
If \(u\) is a kinetic solution of \eqref{f:scalar}, then 
\begin{equation}\label{f:prop19}
\begin{split}
\lim_{R\to+\infty}  \lim_{\epsilon\to 0} & \int_0^T
\int_{\R^{n+1}}\ (\nu*\varphi_\epsilon)
\ \psi_R\, dm_\epsilon 
\\ = {} &  
- \lim_{R\to+\infty} \lim_{\epsilon\to 0}
\int_0^T \int_{\R^{n+1}}\psi_R(x,v) d \Resto[\fe]
\\ 
= {} & 
\int_0^T \int_{\mathcal {N}} Q(u)\,d\mathcal \H^{n-1}\,dt\,,
\end{split}
\end{equation}
where (recall \eqref{shortahand1})
\begin{equation}\label{defQ}
\begin{split}
Q(u)&:=
\chi(u^-, u^+) A^+(u^-)
+ \chi(u^+, u^-) A^+(u^+)
\\ & 
- \chi(u^+, u^-) A^-(u^+) 
- \chi(u^-, u^+) A^-(u^-)
\\ &
- \chi(u^+, \uh) A^+(u^+)
+ \chi(u^+, \uh) A^-(u^+)
\\ & 
- \chi(u^-, \uh) A^+(u^-)
+ \chi(u^-, \uh) A^-(u^-)\,,
\end{split}
\end{equation}
and
\[
\nu(t,x,v):= \delta_{u^+(t,x)}(v) + \delta_{u^-(t,x)}(v)=-2\partial_v f^*\,.
\]
\end{lemma}

\begin{proof}
Let \(\eta_\delta \in C_c^1(0,T)\) with \(\eta=1\) on \([\delta,T-\delta]\).  
Since
\(u\in C([0,T];L^1(\R^n))\), we have that \(f_\eps\in C([0,T];L^1(\R^{n+1}))\), hence 
by testing \eqref{fepsquadro1} and \eqref{feps} with \(\eta_\delta(t)\psi_R(x,v)\) and  letting \(\delta\to 0\) we obtain by standard computations
\begin{align}\label{11}
\int_{\R^{n+1}} & [f^2_\epsilon(T,x,v)-\fe(T,x,v)] \psi_R(x,v)\ dx\,dv
\\
\label{22}
= {} & 
\int_{\R^{n+1}}[f^2_\epsilon(0,x,v)-\fe(0,x,v)]\psi_R(x,v)\ dx\,dv
\\
\label{333} 
&+\int_0^T\int_{\R^{n+1}} (f_\epsilon^2-\fe) \nabla_{x,v}\psi_R\cdot  d\,\a(x,v)\,dt
\\
\label{77}
&+\int_0^T\int_{\R^{n+1}}[-2\partial_v f_\epsilon^*]\psi_R \ dm_\epsilon
\\
\label{88}
&+\int_0^T\int_{\R^{n+1}}[1-2f_\epsilon^*] \partial_v \psi_R\ dm_\epsilon
\\
\label{99}
&+\int_0^T\int_{\R^{n+1}}[
2f_\epsilon^*-1] \psi_R\ dr_\epsilon
\\
\label{unused}
&+\int_0^T \int_{\R^{n+1}}\psi_R\, d \Resto[\fe](x,v)\, dt\,.
\end{align}
As \(\epsilon\) goes to \(0\), the integral \eqref{11} tends  to
\[
\int_{\R^{n+1}}\ 
\big[f^2(T,x,v)-f(T,x,v)\big]
\ \psi_R(x,v)\ dx\,dv=0,
\]
since $f^2=f$. 
In  the same way  the integral \eqref{22} tends to \(0\). Moreover, recalling the definition of   \(\a\) we obtain that 
\begin{equation*}
\begin{split}
\int_0^T&\int_{\R^{n+1}} (f_\epsilon^2-\fe) \nabla_{x,v}\psi_R(x,v)\cdot  d\,\a(x,v)\,dt\\
&=
\int_0^T\int_{\R^{n+1}} 
(f_\epsilon^2-\fe)\nabla_x \psi_R \cdot  \partial_v \A\,dx\,dv\,dt\\
&+
\int_0^T\int_{\R^{n+1}}\ 
(f_\epsilon^2-\fe)\partial_v \psi_R \,d a_{n+1}\,dt\,.
\end{split}
\end{equation*}
We now note that, arguing as the end of the proof of  Lemma \ref{fepsilon},  
$f_\epsilon^2\to f^2=f$, for \(\L_t^1\times(\L^n+\sigma)_x\times \L^1_v\) almost every point. This and the fact that \(|a_{n+1}|\aac \L^n+\sigma\), see  \eqref{compact},  imply that  the integral \eqref{333} tends to $0$, as $\epsilon\to 0$.

Let us now we consider integral \eqref{77}. We remark that 
$$
\partial_v f_\epsilon^*=-\frac{1}{2}\nu*\phi_\epsilon
=-\frac{1}{2}[\delta_{u_-(t,x)}(v)+\delta_{u_+(t,x)}(v)]*\phi_\epsilon.
$$
Therefore the integral \eqref{77} is equal to
\[
\int_0^T\int_{\R^{n+1}}\ 
(\nu*\phi_\epsilon)
\ \psi_R d m_\epsilon.
\]
Let us now estimate the integral \eqref{88}.
Recalling the choice of the test function made in
\eqref{f:testf},
we have that
\[
\left|\int_0^T
\int_{\R^{n+1}}\ 
m_\epsilon(t,x,v)[
1-2f_\epsilon^*]
\ \partial_v \psi_R(x,v)\ dt\,dx\,dv\right|
\leq 2\,m_\epsilon([0,T]\times \R^n\times I_R)\,, 
\]
where $I_R=\{v: R\leq |v|\leq R+1\}\,.$ Hence, by letting   $\epsilon\to 0$ and $R\to+\infty$, the integral \eqref{88}  tends to $0$ since  $m([0,T]\times \R^n\times \R)<+\infty$.

By \eqref{eq:zero}, the integral in \eqref{99} tends to \(0\) as \(\epsilon\to 0\). 
Gathering all the information above, the first equality in \eqref{f:prop19} is proved.

It remains to compute
\[
- \lim_{R\to+\infty} \lim_{\epsilon\to 0}
\int_0^T \int_{\R^{n+1}}\psi_R(x,v) d \Resto[\fe]
\]
and to show that the second equality in \eqref{f:prop19} holds.
We recall that the explicit form of $\Resto[\fe]$ is given
in Lemma~\ref{fepsilonquadro}.
We have that
\[
\begin{split}
- \int_0^T & \int_{\R^{n+1}}\psi_R(x,v) d \Resto[\fe]
\\ = {} &
-\int_0^T \int_{\R^{n+1}} (\fe^+-\fe)(\fe^- - \fe)
(A^+ - A^-)\partial_v \psi_R\, dv\, d\H^{n-1}(x)\, dt
\\ &
-\int_0^T \int_{\R^{n+1}} \big[
(\fe^+-\fe)\partial_v(\fe^- - \fe)
+ (\fe^--\fe)\partial_v(\fe^+ - \fe)
\\ & \hskip2cm +(\fe^+-\fe + \fe^- - \fe)\partial_v \fe
\big]
(A^+ - A^-) \psi_R\, dv\, d\H^{n-1}(x)\, dt
\\ =: {} & I_1(\epsilon, R) + I_2(\epsilon, R)\,.
\end{split}
\]
Again, the first integral $I_1(\epsilon, R)$  tends to $0$ as \(\epsilon\to 0\) and \(R\to +\infty\)
thanks to the choice of the test function $\psi_R$.
In order to compute the $I_2(\epsilon, R)$, let us recall that
\[
\partial_v(\fe^{\pm}-\fe) = 
[\delta_{\uh}(v)-\delta_{u^{\pm}}(v)]\ast
\varphi_{\epsilon}(v)\,,
\qquad
\partial_v\fe
= - \delta_{\uh}(v)\ast \varphi_{\epsilon}(v).
\]
A straightforward computation based on Lemma \ref{tecnico} below, now  gives that
\begin{equation}\label{treno1}
\lim_{R\to+\infty} \lim_{\epsilon\to 0} I_2(\epsilon, R) =
\int_0^T \int_{\mathcal{N}}
Q(u)\, d\H^{n-1}(x)\, dt,
\end{equation}
where
\begin{equation}\label{treno2}
\begin{split}
Q(u) := {} &
\left[-\chi(u^+, \uh) + \chi(u^+, u^-)\right]\, [A^+(u^+) - A^-(u^+)]
\\ & +
\left[\frac{1}{2}-\chi(\uh, u^-) + 
\frac{1}{2} - \chi(\uh, u^+)\right]\, [A^+(\uh) - A^-(\uh)]
\\ & +
\left[-\chi(u^-, \uh) + \chi(u^-, u^+)\right]\, [A^+(u^-) - A^-(u^-)]
\\ & +
\left[-1 +\chi(\uh, u^-) +  \chi(\uh, u^+)\right]\, [A^+(\uh) - A^-(\uh)]\,,
\end{split}
\end{equation}
which, after some computations is easily seen to coincide with \eqref{defQ}. The second equality
in \eqref{f:prop19} now follows form this together with \eqref{treno1} and \eqref{treno2}.
\end{proof}

We conclude this section with the following technical lemma that we have used in the proof of \eqref{treno2}, for later use we state the lemma in a slightly more general setting.
\begin{lemma}\label{tecnico}
Let \(\varphi_\epsilon:\R\to \R\) be a family of symmetric mollifiers and let \(h_1\, h_2:\R \to \R\) be two bounded and uniformly continuous functions, then for every \(u\,,\hat u\in \R\)
\begin{equation}\label{ladri0}
\lim_{\epsilon\to 0}\int
h_1(v)\, (\delta_u\ast\varphi_\epsilon)(v)\,
[h_2(\cdot)\chi(\cdot, \uh)]\ast\varphi_\epsilon(v)\, dv
= h_1(u) \, h_2(u)\, \chi(u, \uh)\,.
\end{equation}
Here \(\chi(u,\uh)\) is defined according to \eqref{defchi}.
\end{lemma}

\begin{proof}
We start noticing that 
\[
\begin{split}
I &:=\int
h_1(v)\, (\delta_u\ast\varphi_\epsilon)(v)\,
[h_2(\cdot)\chi(\cdot, \uh)]\ast\varphi_\epsilon(v)\, dv
\\ &=  
\iint
h_1(v) \varphi_\epsilon(v-u)
h_2(w)\chi(w,\uh)\varphi_\epsilon(v-w)\, dw\, dv
\\  &= 
\iint
h_1(v) \varphi_\epsilon(v-u)
[h_2(w) - h_2(u)]\chi(w,\uh)\varphi_\epsilon(v-w)\, dw\, dv
\\ & \quad+
h_2(u)\, \iint
h_1(v) \varphi_\epsilon(v-u)
\chi(w,\uh)\varphi_\epsilon(v-w)\, dw\, dv
\\& =:   I_1 + h_2(u)\, I_2\,.
\end{split}
\]
By exploiting the uniform continuity of \(h_2\) we obtain for some modulus of continuity \(\omega\), that 
\begin{equation}\label{ladri}
\begin{split}
|I_1| & \leq  \iint |h_1(v)|\, \varphi_\epsilon(v-u)\,
\omega(|v-w|)\, \varphi_\epsilon(v-w)\, dw\, dv
\\ & \leq  \|h_1\|_{\infty}
\int \left( \omega (|s|) \varphi_\epsilon(s)\, \int \varphi_\epsilon(s+w-u) dw\right)dt
\\ & \le
\|h_1\|_{\infty}
\int \omega(|s|) \varphi_\epsilon(s)\,ds
\end{split}
\end{equation}
and the last integral tends to $0$ as $\epsilon\to 0$. On the other hand
\begin{equation}\label{ladri2}
\lim_{\epsilon\to 0+} I_2
= h_1(u) \chi(u,\uh)\,.
\end{equation}
Indeed, from the estimate
\[
\begin{split}
|I_2 - h_1(u)\chi(u, \uh)| \leq {} &
\, \max_{v\in [u-\epsilon, u+\epsilon]} |h_1(v) - h_1(u)|
\\ & + |h_1(u)|  
\left|\iint \varphi_\epsilon(v-u)\chi(w,\uh)\varphi_\epsilon(v-w)\, dw dv - \chi(u, \uh)\right|
\end{split}
\]
and the uniform continuity of $h_1$
it is enough to show that
\begin{equation}\label{ladri3}
\lim_{\epsilon\to 0+} 
\iint \varphi_\epsilon(v-u)\chi(w,\uh)\varphi_\epsilon(v-w)\, dw dv = \chi(u, \uh)\,.
\end{equation}
If $u\neq \uh$, then for $\epsilon$ small enough we have that
$\chi(w, \uh) = \chi(u,\uh)$ for $w\in (u-2\epsilon, u+2\epsilon)$, so that $\chi$ is constant in the integral.
Then
\[
\iint \varphi_\epsilon(v-u)\chi(w,\uh)\varphi_\epsilon(v-w)\, dw dv
=
\chi(u,\uh)\iint \varphi_\epsilon(v-u)\varphi_\epsilon(v-w)\, dw dv = \chi(u, \uh)
\]
and \eqref{ladri3} follows. If $u=\uh$,
then the integral in \eqref{ladri3}, for $\epsilon > 0$
small enough, becomes
\[
\begin{split}
& \iint \varphi_\epsilon(v-u)\chi(w,u)\varphi_\epsilon(v-w)\, dw dv
= 
\int_{-\infty}^{+\infty} \varphi_\epsilon (v-u)
\int_{-\infty}^{u}\varphi_\epsilon(v-w)\, dw \, dv
\\ & =
\int_{-\infty}^{+\infty} \varphi_\epsilon (v-u)
\left[\frac{1}{2} + 
\int_{v}^{u}\varphi_\epsilon(v-w)\, dw \right] dv
= \frac{1}{2}\,,
\end{split}
\] 
where we have exploited that \(\varphi_\epsilon(s)=\varphi_\epsilon(-s)\). 
Equation~\eqref{ladri3} now follows form the definition of \(\chi\), \eqref{defchi}. Hence \eqref{ladri0} follows from \eqref{ladri} and \eqref{ladri2}.
\end{proof}


\section{Uniqueness}
\label{s:uniqueness}
In this Section we prove Theorem \ref{thm:uniquenessentropic}, to this end let us fix some notation. For $u_1, u_2$  two entropy solutions of \eqref{f:scalar},
with corresponding everywhere defined Borel representatives $\uh_1, \uh_2$, we will set  $f_i(t,x,v) := \chi(v,\uh_i(t,x))$ and $m_i$, $i=1,2$, for the corresponding  functions and measures appearing in the kinetic formulation. We will also set 
\[
f_{1\epsilon}(t,x,v) := (f_1(t,x,\cdot)\ast\varphi_{\epsilon})(v),
\qquad
f_{2\epsilon}(t,x,v) := (f_2(t,x,\cdot)\ast\varphi_{\epsilon})(v)\,.
\]
and
\[
m_{1\epsilon}(t,x,v) := (m_1(t,x,\cdot)\ast\varphi_{\epsilon})(v),
\qquad
m_{2\epsilon}(t,x,v) := (m_2(t,x,\cdot)\ast\varphi_{\epsilon})(v),
\]
where
$\varphi_{\epsilon}$ denotes the standard family of
mollifiers.

The following theorem is the main result of this section and,  by Cavalieri's principle, it immediately  implies Theorem \ref{thm:uniquenessentropic}.
\begin{theorem}\label{uniqueness}
Let $u_1, u_2$ be two 
entropy solution of \eqref{f:scalar},
with corresponding everywhere defined Borel 
representatives $\uh_1, \uh_2$.
Setting $f_i(t,x,v) := \chi(v,\uh_i(t,x))$, $i=1,2$, we have that
\[
\int_{\R^{n+1}}\ 
|f_1-f_2|(T,x,v) 
\ dx\,dv 
\leq \int_{\R^{n+1}}\ 
|f_1-f_2|(0,x,v)
\ dx\,dv \,.
\]
\end{theorem}

In order to prove Theorem \ref{uniqueness} we need some preliminary results concerning the interaction of two kinetic solutions:

\begin{proposition} 
With the notation above, for every test function $\psi(x,v)$ we have that
\begin{equation}\label{f:contr1}
\begin{split}
\int_{\R^{n+1}}\ 
|f_1-f_2|(T,x,v)&
\, \psi(x,v)\ dx\,dv 
\leq \int_{\R^{n+1}}\ 
|f_1-f_2|(0,x,v)
\, \psi(x,v)\ dx\,dv 
\\ & 
-\limsup_{\epsilon\to0}\int_0^T \int_{\R^{n+1}} 
2\partial_v[(f_{1\epsilon}-f_{2\epsilon})^*\psi]
d(m_{1\epsilon}-m_{2\epsilon})(t,x,v)
\\&
+\limsup_{\epsilon\to0}\int_0^T \int_{\R^{n+1}}\psi\, d\Resto[f_{1\epsilon}-f_{2\epsilon}](t,x,v)
\\ & -
\int_0^T \int_{\R^{n+1}} |f_1-f_2|\,
\nabla_{x}\psi\cdot \partial_v \A(x,v)\, dx\, dv\, dt
\\ &+
 \int_0^T \int_{\R^{n+1}} |f_1-f_2|\,
\partial_{v}\psi\, da_{n+1}(x,v)\, dt
\,,
\end{split}
\end{equation}
where \(\Resto[f_{1\epsilon}-f_{2\epsilon}]\) is defined according to \eqref{resto0}.
\end{proposition}

\begin{proof}
We recall that
$$
Lf_{1\epsilon}=\partial_{v}m_{1\epsilon}+r_{1\epsilon}\,,\qquad\ Lf_{2\epsilon}=\partial_{v}m_{2\epsilon}+r_{2\epsilon}\,.
$$ 
Since  $f_{1\epsilon}-f_{2\epsilon}$ is a solution of \eqref{fepsquadro1}, by Lemma \ref{fepsilonquadro} we see that
\[
\begin{split}
&\partial_t(f_{1\epsilon}-f_{2\epsilon})^2+
\Div_{x,v}[\a(f_{1\epsilon}-f_{2\epsilon})^2]\cr
& = 2(f_{1\epsilon}-f_{2\epsilon})^*L(f_{1\epsilon}-f_{2\epsilon})+\Resto[f_{1\epsilon}-f_{2\epsilon}]\,.
\end{split}
\]
Hence, by arguing as in Lemma \ref{prop.19},  for every test function $\psi(x,v)$ we have that 
\begin{equation}\label{110}
\begin{split}
\int_{\R^{n+1}} & (f_{1\epsilon}-f_{2\epsilon})^2(T,x,v)\psi(x,v)dx\,dv-
\int_{\R^{n+1}}(f_{1\epsilon}-f_{2\epsilon})^2(0,x,v)\psi(x,v)dx\,dv
\\ = {} & -\int_0^T\int_{\R^{n+1}} (f_{1\epsilon}-f_{2\epsilon})^2\, \nabla_{x,v} \psi(x,v) \cdot d \a(x,v)\,dt
\\ & +
2\int_0^T\int_{\R^{n+1}}\ 
(f_{1\epsilon}-f_{2\epsilon})^*L(f_{1\epsilon}-f_{2\epsilon})
\ \psi(x,v)\ dx\,dv\,dt
\\ & +
\int_0^T\int_{\R^{n+1}}\ 
\ \psi(x,v)\ d (\Resto[f_{1\epsilon}-f_{2\epsilon}])
\\ =: {} & I_1 + I_2 + I_3\,.
\end{split}
\end{equation}
Noticing that $(f_{1\epsilon}-f_{2\epsilon})^2\to (f_1 - f_2)^2$, 
that the equality
$(f_1 - f_2)^2=|f_1 - f_2|$  
holds for \(\L^1_t\times (\L^n+\sigma)\times \L^1_v\) 
almost every point, 
and that \(|\a|\aac (\L^n+\sigma)\times \L^1_v\), we can pass to the limit in \(I_1\) 
obtaining the last two integrals in the right hand side of \eqref{f:contr1}.
In the same way, the left hand side of  \eqref{110} tends, as $\epsilon\to 0$,  to
\[
\int_{\R^{n+1}}\ 
|f_{1}-f_2|(T,x,v)\psi(x,v)\,dx\,dv-
\int_{\R^{n+1}}\ 
|f_{1}-f_{2}|(0,x,v)\psi(x,v)
\ dx\,dv\,.
\]
Let us now consider \(I_2\) and \(I_3\). 
Since, by Proposition  \ref{fepsilon}, 
\[
Lf_{1\epsilon}=\partial_{v}m_{1\epsilon}+r_{1\epsilon}\,\,,\quad\ Lf_{2\epsilon}=\partial_{v}m_{2\epsilon}+r_{2\epsilon}
\] 
we infer that inequality \eqref{f:contr1} holds.
\end{proof}

\begin{proposition}\label{uniqueness1}
Let $u_1, u_2$ be two 
entropy solution of \eqref{f:scalar},
with corresponding representatives $\uh_1, \uh_2$.
Setting $f_i(t,x,v) := \chi(v,\uh_i(t,x))$, $i=1,2$, we have that

\begin{equation}\label{f:contr}
\begin{split}
\int_{\R^{n+1}}\ 
|f_1-f_2|(T,x,v) 
\ dx\,dv 
&\leq \int_{\R^{n+1}}\ 
|f_1-f_2|(0,x,v)
\ dx\,dv \\
&+\int_0^T \int_{\mathcal N} \W(u_1,u_2)d\H^{n-1}dt
\,,
\end{split}
\end{equation}
where 
\begin{equation}\label{57bis}
\begin{split}
 \W(u_1,u_2)
  := {} 
& A^+(u_1^+)\big[-2\chi(u_1^+,u_2^+)+2\chi(u_1^-,u_2^-)\big]\\
+ &
A^+(u_2^+)\big[-2\chi(u_2^+,u_1^+)+2\chi(u_2^-,u_1^-)\big]\,.
\end{split}
\end{equation}
\end{proposition}

\begin{proof}
Let us start from inequality \eqref{f:contr1} using the
test function $\psi = \psi_R$ defined in \eqref{f:testf}.
It is easy to show that, as $R\to +\infty$,
the last two integrals in \eqref{f:contr1} goes to $0$. Hence, we only need to estimate the contribution of 
 the two terms
\[
I(\epsilon,R) := -\int_0^T \int_{\R^{n+1}} 
2\partial_v\big[(f_{1\epsilon}-f_{2\epsilon})\psi_R\big]
d(m_{1\epsilon}-m_{2\epsilon})(t,x,v)
\]
and
\[
H(\epsilon,R) := \int_0^T \int_{\R^{n+1}}\psi_R\, d(\Resto[f_{1\epsilon}-f_{2\epsilon}])(t,x,v)
\]
for $\epsilon\to 0$ and $R\to +\infty$. We start noticing that 
\begin{equation}\label{f:eqI}
\begin{split}
I(\epsilon,R) = {} & -\int_0^T \int_{\R^{n+1}} 
2\partial_v[(f_{1\epsilon}-f_{2\epsilon})^*]\psi_R
d(m_{1\epsilon}-m_{2\epsilon})(t,x,v)
\\ & -\int_0^T \int_{\R^{n+1}} 
2(f_{1\epsilon}-f_{2\epsilon})^*\partial_v \psi_R
d(m_{1\epsilon}-m_{2\epsilon})(t,x,v)\,,
\end{split}
\end{equation}
and,
reasoning as
in the final part of the proof of Lemma~\ref{prop.19},
the last integral goes to $0$ as 
$\epsilon\to 0$ and $R\to +\infty$. 
Let us compute the first term in \eqref{f:eqI}:
\[
\begin{split}
\\  
- \int_0^T & \int_{\R^{n+1}} 
2\partial_v[(f_{1\epsilon}-f_{2\epsilon})^*]\psi_R
d(m_{1\epsilon}-m_{2\epsilon})
\\ = {} &
-\int_0^T \int_{\R^{n+1}} 
[\delta_{u^+_2}(v)+\delta_{u^-_2}(v)-\delta_{u^+_1}(v)-\delta_{u^-_1}(v)
]*\varphi_\epsilon(v)
\,\psi_R\,
d(m_{1\epsilon}-m_{2\epsilon})
\\ = {} &
\int_0^T \int_{\R^{n+1}} 
[\delta_{u^+_1}(v)+\delta_{u^-_1}(v)
]*\varphi_\epsilon(v)
\,\psi_R\,
dm_{1\epsilon}
\\&
+
\int_0^T \int_{\R^{n+1}} 
[\delta_{u^+_2}(v)+\delta_{u^-_2}(v)
]*\varphi_\epsilon(v)
\,\psi_R\,
dm_{2\epsilon}
\\&
-
\int_0^T \int_{\R^{n+1}} 
[\delta_{u^+_2}(v)+\delta_{u^-_2}(v)
]*\varphi_\epsilon(v)
\,\psi_R\,
dm_{1\epsilon}
\\&
-
\int_0^T \int_{\R^{n+1}} 
[\delta_{u^+_1}(v)+\delta_{u^-_1}(v)
]*\varphi_\epsilon(v)
\,\psi_R\,
dm_{2\epsilon}
\\ =: {} & 
I_1(\epsilon,R)+I_2(\epsilon,R)+I_3(\epsilon,R)+I_4(\epsilon,R)\,.
\end{split}
\]
By Lemma \ref{prop.19},
\[
\begin{split}
\limsup_{R\to \infty}&\limsup_{\epsilon\to 0}\,  I_1(\epsilon,R)+I_2(\epsilon,R)
\\
&\le \int_0^T \int_{\mathcal {N}} [Q(u_1)+Q(u_2)]\,d\mathcal \H^{n-1} \,dt\,,
\end{split}
\]
where \(Q(u)\) is defined in \eqref{defQ}. In order to estimate $I_3$, we note that  
\[
\begin{split}
I_3 & = -
\int_0^T \int_{\R^{n+1}} 
[\delta_{u^+_2}(v)+\delta_{u^-_2}(v)
]*\varphi_\epsilon(v)
\,\psi_R\,
dm_{1\epsilon}(t,x,v)
\\&
\leq-
\int_0^T \int_{\R^{n+1}} 
[\delta_{u^+_2}(v)+\delta_{u^-_2}(v)
]*\varphi_\epsilon(v)
\,\psi_R\,
dm^s_{1\epsilon}(t,x,v)\,,
\end{split}
\]
where
$m^s_{1\epsilon}(t,\cdot,v)$ is 
the restriction to $\mathcal{N}$
of the singular part 
of the measure $m_{1\epsilon}(t,\cdot,v)$, namely
$m^s_{1\epsilon}(t,\cdot,v) = m_1^s(t,\cdot,v)
* \varphi_{\epsilon}(v)$,
with
\[
\begin{split}
m_1^s(t,x,v) = &
\int_{0}^{v}\left[
\left(\partial_v A\right)^{+}(x,w)
(\chi(w,u_1^{+}(t,x)))
-\left(\partial_v A\right)^{-}(x,w)(\chi(w,u_1^{-}(t,x)))
\right]
d\H^{n-1}\,dw
\\ & 
-\left(A^{+}(x,v)-A^{-}(x,v)\right)
\chi(v,\uh_1(t,x))
d\H^{n-1}.
\end{split}
\]
Hence, the contribution given by $I_3$ to
\eqref{f:contr} can be estimated by
\begin{equation*}
\begin{split}
   -\limsup_{R\to+\infty} \limsup_{\epsilon\to 0}
\int_0^T \int_{\R^{n+1}} 
[\delta_{u^+_2}(v)+\delta_{u^-_2}(v)
]*\varphi_\epsilon(v)
\,\psi_R\,
dm_{1\epsilon}^s\,.
\end{split}
\end{equation*}
By the explicit expression of  $m_{1\epsilon}^s$ and by also taking into account  Lemma~\ref{tecnico}, we get
\begin{equation*}
\begin{split}
I_3 \le  {} & \int_0^T \int_{\mathcal {N}}\Big[
-\chi(u^+_2, u^+_1) A^+(u^+_2)
- \chi(u^+_1, u^+_2) A^+(u^+_1)
\\
 & \quad+
\chi(u^+_2, u^-_1) A^-(u^+_2)
+ \chi(u^-_1, u^+_2) A^-(u^-_1)
\\
 &\quad -
\chi(u^+_2, \uh_1) A^-(u^+_2)
+ \chi(u^+_2, \uh_1) A^+(u^+_2)\Big]\,
d\H^{n-1}
\\
 & +  \int_0^T \int_{\mathcal {N}}\Big[
-\chi(u^-_2, u^+_1) A^+(u^-_2)
- \chi(u^+_1, u^-_2) A^+(u^+_1)
\\
 &\quad +
\chi(u^-_2, u^-_1) A^-(u^-_2)
+ \chi(u^-_1, u^-_2) A^-(u^-_1)
\\
 & \quad -
\chi(u^-_2, \uh_1) A^-(u^-_2)
+ \chi(u^-_2, \uh_1) A^+(u^-_2)\Big]\,
d\H^{n-1}\,.
\end{split}
\end{equation*}
In the same way we can estimate the  contribution of  $I_4$. The contribution of \(H\) can then be estimated with the aid of Lemma \ref{tecnico}  by arguing as in the final part of the proof of Lemma \ref{prop.19}.  Therefore, by summing up all the estimates  we obtain that  \eqref{f:contr} holds with the following choice of \(\W(u_1,u_2)\) where, for the reader convenience,  the terms are grouped according to their provenience:  
\begin{align*}
& \W(u_1,u_2)\\
& \left.
\begin{aligned} 
&= \chi(u_1^-, u_1^+) A^+(u_1^-)
+ \chi(u_1^+, u_1^-) A^+(u_1^+)
- \chi(u_1^+, u_1^-) A^-(u_1^+)
- \chi(u_1^-, u_1^+) A^-(u_1^-)
\\ & 
- \chi(u_1^+, \uh_1) A^+(u_1^+)
+ \chi(u_1^+, \uh_1) A^-(u_1^+)
- \chi(u_1^-, \uh_1) A^+(u_1^-)
+ \chi(u_1^-, \uh_1) A^-(u_1^-)
\end{aligned}
\quad
\right\}I_1 &
\\
\\
& \left.
\begin{aligned}
& + \chi(u_2^-, u_2^+) A^+(u_2^-)
+ \chi(u_2^+, u_2^-) A^+(u_2^+)
- \chi(u_2^+, u_2^-) A^-(u_2^+)
- \chi(u_2^-, u_2^+) A^-(u_2^-)
\\ & 
- \chi(u_2^+, \uh_2) A^+(u_2^+)
+ \chi(u_2^+, \uh_2) A^-(u_2^+)
- \chi(u_2^-, \uh_2) A^+(u_2^-)
+ \chi(u_2^-, \uh_2) A^-(u_2^-)
\end{aligned}
\quad
\right\} I_2&
\\
\\
& \left.
\begin{aligned}
 &-\chi(u_1^+,u_2^+)A^{+}(u_1^+)-\chi(u_1^-,u_2^+)A^{+}(u_1^-)
-
\chi(u_2^+,u_1^+)A^{+}(u_2^+)-\chi(u_2^+,u_1^-)A^{+}(u_2^+)
\\ &
+
\chi(u_1^+,u_2^-)A^{-}(u_1^+)+\chi(u_1^-,u_2^-)A^{-}(u_1^-)
+
\chi(u_2^-,u_1^+)A^{-}(u_2^-)+\chi(u_2^-,u_1^-)A^{-}(u_2^-)
\\ &
+
\chi(u_1^+,\uh_2)A^{+}(u_1^+)-\chi(u_1^+,\uh_2)A^{-}(u_1^+)
+
\chi(u_1^-,\uh_2)A^{+}(u_1^-)-\chi(u_1^-,\uh_2)A^{-}(u_1^-)
\end{aligned}
\quad
\right\} I_3&
\\
\\
& \left.
\begin{aligned}
& -
\chi(u_2^+,u_1^+)A^{+}(u_2^+)-\chi(u_2^-,u_1^+)A^{+}(u_2^-)
-
\chi(u_1^+,u_2^+)A^{+}(u_1^+)-\chi(u_1^+,u_2^-)A^{+}(u_1^+)
\\ & 
+
\chi(u_2^+,u_1^-)A^{-}(u_2^+)+\chi(u_2^-,u_1^-)A^{-}(u_2^-)
+
\chi(u_1^-,u_2^+)A^{-}(u_1^-)+\chi(u_1^-,u_2^-)A^{-}(u_1^-)
\\ &
+
\chi(u_2^+,\uh_1)A^{+}(u_2^+)-\chi(u_2^+,\uh_1)A^{-}(u_2^+)
+
\chi(u_2^-,\uh_1)A^{+}(u_2^-)-\chi(u_2^-,\uh_1)A^{-}(u_2^-)
\end{aligned}
\quad
\right\} I_4&
\\
\\
& \left.
\begin{aligned}
& - [\chi(u_1^+,u_1^-)-\chi(u_1^+,\uh_1)-\chi(u_1^+,u_2^-)+\chi(u_1^+,\uh_2)]
\big[A^{+}(u_1^+) - A^-(u_1^+)\big]
\\ & +[\chi(u_2^+,u_1^-)-\chi(u_2^+,\uh_1)-\chi(u_2^+,u_2^-)+\chi(u_2^+,\uh_2)]
\big[A^{+}(u_2^+) - A^-(u_2^+)\big]
\\ & -[\chi(u_1^-,u_1^+)-\chi(u_1^-,\uh_1)-\chi(u_1^-,u_2^+)+\chi(u_1^-,\uh_2)]
\big[A^{+}(u_1^-) - A^-(u_1^-)\big]
\\ & +[\chi(u_2^-,u_1^+)-\chi(u_2^-,\uh_1)-\chi(u_2^-,u_2^+)+\chi(u_2^-,\uh_2)]
\big[A^{+}(u_2^-) - A^-(u_2^-)\big]\,.
\end{aligned}
\quad
\right\} H&
\end{align*}
We now note that  all terms in
$A^+(u_j^-)$ and $A^-(u_j^+)$, $j=1,2$,
cancel out. Using the Rankine-Hugoniot conditions \eqref{eq:RH2},  we can sum up the terms $A^+(u_j^+)$ with $A^-(u_j^-)$, 
obtaining (\ref{57bis}).
\end{proof}

We are now ready to prove Theorem \ref{uniqueness}.
\begin{proof}[Proof  of Theorem \ref{uniqueness}]
By Proposition \ref{uniqueness1}, it remains to prove that
\begin{equation}\label{f:restototale}
A^+(u_1^+)\big[-\chi(u_1^+,u_2^+)+\chi(u_1^-,u_2^-)\big]+
A^+(u_2^+)\big[-\chi(u_2^+,u_1^+)+\chi(u_2^-,u_1^-)\big]\leq 0\,.
\end{equation}
Since $u_1$, $u_2$ are entropy solutions,
recalling that
\[
\chi(a,b) + \chi(b,a) = 1
\qquad\forall\, a,b\in\R,
\]
in order to prove  \eqref{f:restototale}, we have to show that
\begin{equation}\label{f:restopos}
\W(u_1,u_2) = (\chi^+ - \chi^-) \left[
A^+(u_2^+) - A^+(u_1^+)\right] \leq 0\,,
\end{equation}
where we have set  $\chi^{\pm} = \chi(u_1^{\pm}, u_2^{\pm})$.

If $\sign(u_2^+-u_1^+) = \sign(u_2^--u_1^-)$, then 
$\chi^+ = \chi^-$, and   \eqref{f:restopos} is
satisfied. Hence we shall restrict our attention to the case
$1 = \sign(u_2^+-u_1^+) = -\sign(u_2^--u_1^-)$, i.e.
\begin{equation}\label{f:case1}
u_1^+ < u_2^+, \qquad
u_2^- < u_1^-,
\end{equation}
so that $\chi^+ = 1$, $\chi^- = 0$,
since the case $-1 = \sign(u_2^+-u_1^+) = -\sign(u_2^--u_1^-)$
can be handled in a similar way.

Since \eqref{f:case1} is in force in order to prove \eqref{f:restopos} we need to show that 
\begin{equation}\label{f:restopos2}
A^+(u_2^+) - A^+(u_1^+) \leq 0
\quad\text{or, equivalently,}\quad
A^-(u_2^-) - A^-(u_1^-) \leq 0
\end{equation}
in each one of
the following six possibilities:
\begin{gather}
u_1^+ < u_2^+ < u_2^- < u_1^-, \label{f:c1}\\
u_1^+ < u_2^-  < u_2^+ < u_1^-, \label{f:c2}\\
u_2^- < u_1^+  < u_2^+ < u_1^-, \label{f:c3}\\
u_2^- < u_1^+  < u_1^- < u_2^+, \label{f:c4}\\
u_2^- < u_1^-  < u_1^+ < u_2^+, \label{f:c5}\\
u_1^+ < u_2^-  < u_1^- < u_2^+. \label{f:c6}
\end{gather}
Here, for simplicity, we have considered only strict inequalities,
but the equality cases can be handled as well.
Namely, if $u_1^-=u_2^-$ or $u_1^+=u_2^+$, then
\eqref{f:restopos2} trivially holds,
while the other equality cases can be proved using a
continuity argument,
since the quantities $\chi^{\pm}$ 
do not change. 

The analysis will be the same as the one performed 
in \cite{Mitr}.
For the reader's convenience, we briefly report here
how to use the inequalities
\eqref{Subcase1a(i)}--\eqref{Subcase2c(vi)}
in order to prove \eqref{f:restopos2} in each of the
six cases listed above.
When not explicitly stated, the inequalities
\eqref{Subcase1a(i)}--\eqref{Subcase2c(vi)} are 
supposed to be used
with $u=u_1$.

\medskip
\noindent
\(\bullet\)
Case \eqref{f:c1}.
We have the following possibilities:
\begin{itemize}
\item
$u_1^+ < u_2^+ < u_2^- < u_1^- < \uh_1$:
choose $c = u_2^+$ in \eqref{Subcase1a(ii)}.
\item
$u_1^+ < u_2^+ < \uh_1 < u_1^-$:
choose $c = u_2^+$ in \eqref{Subcase1b(iv)}.
(We remark that the relative position of $\uh_1$ and $u_2^-$ is
not relevant, so that this case includes the two subcases
$u_1^+ < u_2^+ < \uh_1 < u_2^- < u_1^-$ and
$u_1^+ < u_2^+ < u_2^- <\uh_1 < u_1^-$;
similar considerations will apply also in some 
of the following
cases without further reference.) 
\item
$u_1^+ < \uh_1 < u_2^-  < u_1^- $:
choose $c = u_2^-$ in \eqref{Subcase1b(iii)}.
\item
$\uh_1 < u_1^+ < u_2^+ < u_2^- < u_1^-$:
choose $c = u_2^+$ in \eqref{Subcase1c(v)}.
\end{itemize}

\medskip
\noindent
\(\bullet\)
Case \eqref{f:c2}.
We have the following possibilities:
\begin{itemize}
\item
$u_1^+ < u_2^- < u_2^+ < u_1^- < \uh_1$:
choose $c=u_2^+$ in \eqref{Subcase1a(ii)}.
\item
$u_1^+ < u_2^- < u_2^+ < \uh_1 < u_1^-$:
choose $c=u_2^+$ in \eqref{Subcase1b(iv)}.
\item
$u_1^+ < \uh_1 <  u_2^- < u_2^+ < u_1^-$:
choose $c=u_2^-$ in \eqref{Subcase1b(iii)}.
\item
$\uh_1 < u_1^+ < u_2^- < u_2^+ < u_1^-$:
choose $c=u_2^+$ in \eqref{Subcase1c(v)}.
\item
$u_1^+ < u_2^- <\uh_1 < u_2^+ < u_1^-$:
this case is slightly more involved because we need
to consider also the position of $\uh_2$.
We have the following subcases:
\begin{enumerate}
\item
$u_1^+ < u_2^- <\uh_1 < u_2^+ < \uh_2$:
apply \eqref{Subcase2a(ii)} to $u_2$ with $c=\uh_1$ and then
\eqref{Subcase1b(iii)} to $u_1$ again with $c = \uh_1$,
obtaining
\[
A^-(u_2^-) \leq A^-(\uh_1) \leq A^-(u_1^-).
\]
\item
$u_1^+ < u_2^- <\uh_1 < \uh_2 < u_2^+ < u_1^-$:
apply \eqref{Subcase2b(iv)} to $u_2$ with $c=\uh_1$ and then
\eqref{Subcase1b(iii)} to $u_1$ again with $c = \uh_1$
obtaining
\[
A^-(u_2^-) \leq A^-(\uh_1) \leq A^-(u_1^-).
\]
\item
$u_1^+ < u_2^- <\uh_2 < \uh_1 < u_2^+ < u_1^-$:
apply \eqref{Subcase2b(iii)} to $u_2$ with $c=\uh_1$ and then
\eqref{Subcase1b(iv)} to $u_1$ again with $c = \uh_1$
obtaining
\[
A^+(u_2^+) \leq A^+(\uh_1) \leq A^+(u_1^+).
\]
\item
$\uh_2 < u_2^- < \uh_1 < u_2^+ < u_1^-$:
apply \eqref{Subcase2c(v)} to $u_2$ with $c=\uh_1$ and then
\eqref{Subcase1b(iv)} to $u_1$ again with $c = \uh_1$
obtaining
\[
A^+(u_2^+) \leq A^+(\uh_1) \leq A^+(u_1^+).
\]
\end{enumerate}
\end{itemize}

\medskip
\noindent
\(\bullet\)
Case \eqref{f:c3}.
We have the following possibilities:
\begin{itemize}
\item
$u_2^- < u_1^+ < u_2^+ < u_1^- < \uh_1$:
choose $c=u_2^+$ in \eqref{Subcase1a(ii)}.
\item
$u_2^- < u_1^+ < u_2^+ < \uh_1 < u_1^-$:
choose $c=u_2^+$ in \eqref{Subcase1b(iv)}.
\item
$\uh_1 < u_1^+ < u_2^+ < u_1^-$:
choose $c=u_2^+$ in \eqref{Subcase1c(v)}.
\item
$u_2^- < u_1^+ < \uh_1 <  u_2^+ < u_1^-$:
this case is slightly more involved because we need
to consider also the position of $\uh_2$.
We have the following subcases:
\begin{enumerate}
\item
$u_2^- < u_1^+ < \uh_1 <  u_2^+ < \uh_2$:
apply \eqref{Subcase2a(ii)} to $u_2$ with $c=\uh_1$ and then
\eqref{Subcase1b(iii)} to $u_1$ again with $c = \uh_1$
obtaining
\[
A^-(u_2^-) \leq A^-(\uh_1) \leq A^-(u_1^-).
\]
\item
$u_2^- < u_1^+ < \uh_1 < \uh_2 < u_2^+ < u_1^-$:
apply \eqref{Subcase2b(iv)} to $u_2$ with $c=\uh_1$ and then
\eqref{Subcase1b(iii)} to $u_1$ again with $c = \uh_1$
obtaining
\[
A^-(u_2^-) \leq A^-(\uh_1) \leq A^-(u_1^-).
\]
\item
$u_2^-  < \uh_2 < \uh_1 < u_2^+ < u_1^-$:
apply \eqref{Subcase2b(iii)} to $u_2$ with $c=\uh_1$ and then
\eqref{Subcase1b(iv)} to $u_1$ again with $c = \uh_1$
obtaining
\[
A^+(u_2^+) \leq A^+(\uh_1) \leq A^+(u_1^+).
\]
\item
$\uh_2 < u_1^+ < \uh_1 < u_2^+ < u_1^-$:
apply \eqref{Subcase2c(v)} to $u_2$ with $c=\uh_1$ and then
\eqref{Subcase1b(iv)} to $u_1$ again with $c = \uh_1$
obtaining
\[
A^+(u_2^+) \leq A^+(\uh_1) \leq A^+(u_1^+).
\]
\end{enumerate}
\end{itemize}

\medskip
\noindent
\(\bullet\)
Case \eqref{f:c4}.
This case is symmetric to \eqref{f:c2}. It is enough to
replace $\uh_1$ with $\uh_2$ and to apply 
\eqref{Subcase2a(i)}--\eqref{Subcase2c(vi)}
instead of
\eqref{Subcase1a(i)}--\eqref{Subcase1c(vi)}
or conversely.

\medskip
\noindent
\(\bullet\)
Case \eqref{f:c5}.
This case is symmetric to \eqref{f:c1}. It is enough to
replace $\uh_1$ with $\uh_2$ and to apply 
\eqref{Subcase2a(i)}--\eqref{Subcase2c(vi)}
instead of
\eqref{Subcase1a(i)}--\eqref{Subcase1c(vi)}.

\medskip
\noindent
\(\bullet\)
Case \eqref{f:c6}.
This case is symmetric to \eqref{f:c3}. It is enough to
replace $\uh_1$ with $\uh_2$ and to apply 
\eqref{Subcase2a(i)}--\eqref{Subcase2c(vi)}
instead of
\eqref{Subcase1a(i)}--\eqref{Subcase1c(vi)}
or conversely.
\end{proof}


\def\cprime{$'$}
\providecommand{\bysame}{\leavevmode\hbox to3em{\hrulefill}\thinspace}
\providecommand{\MR}{\relax\ifhmode\unskip\space\fi MR }
\providecommand{\MRhref}[2]{%
  \href{http://www.ams.org/mathscinet-getitem?mr=#1}{#2}
}
\providecommand{\href}[2]{#2}

\end{document}